\documentclass[a4paper,twoside,10pt]{amsart}
\usepackage[english]{babel}
\usepackage{t1enc}
\usepackage[charter]{mathdesign}
\usepackage{amsthm}
\usepackage{esint}
\newcommand*\Ccal{\mathcal C}
\newcommand*\Mcal{\mathcal M}
\newcommand*\Pcal{\mathcal P}
\newcommand*\nsubset{\not\subset}
\newcommand*\nequiv{\not\equiv}
\newcommand*\Nbb{\mathbb N}
\newcommand*\Qbb{\mathbb Q}
\newcommand*\Rbb{\mathbb R}
\newcommand*\Cbb{\mathbb C}
\newcommand*\loc{\mathrm{loc}}
\newcommand*\PI{\mathrm{PI}}
\newcommand*\dbl{\mathrm{dbl}}
\newcommand*\lsc{\mathrm{lsc}}
\newcommand*\emb{\hookrightarrow}
\newcommand*\fcrim{\hspace{0.08333em}}
\newcommand*\Lip{\mathrm{Lip}}
\newcommand*\Lipc{\mathrm{Lip_c}}
\newcommand*\AC{\mathrm{AC}}
\newcommand*\dd{\mathrm{d}}
\newcommand*\cconc{c_{\scriptscriptstyle{\!\vartriangle}}}
\newcommand*\cemb{c_{\mathrm{emb}}}
\newcommand*\upto{\nearrow}
\newcommand*\downto{\searrow}
\newcommand*{\limplus}{{\mathchoice{\raise.17ex\hbox{$\scriptstyle +$}}
                {\raise.17ex\hbox{$\scriptstyle +$}}
                {\raise.1ex\hbox{$\scriptscriptstyle +$}}
                {\scriptscriptstyle +}}}
\newcommand*\NX{{N^1\!X}}
\newcommand*\NnX{{N_0^1 X}}
\newcommand*\NtX{{\widetilde{N}^1\!X}}
\newcommand*\eps{\varepsilon}
\newcommand*\Mod{\mathop{\mathrm{Mod}}\nolimits}
\newcommand*\diam{\mathop{\mathrm{diam}}\nolimits}
\newcommand*\dist{\mathop{\mathrm{dist}}\nolimits}
\newcommand*\spt{\mathop{\mathrm{spt}}\nolimits}
\newcommand*\sgn{\mathop{\mathrm{sgn}}\nolimits}
\newcommand*\esssup{\mathop{\mathrm{ess\,sup}}}
\newcommand*\ri{r.i.\@ }
\newcommand*{\coloneq}{\mathrel{\vcenter{\baselineskip0.5ex \lineskiplimit0pt\hbox{\scriptsize.}\hbox{\scriptsize.}}}=}
\renewcommand*{\cdots}{\mathinner{\vcenter{\hbox{$\ldots$}}}}
\newcommand*{\eqcolon}{=\mathrel{\vcenter{\baselineskip0.5ex \lineskiplimit0pt
                     \hbox{\scriptsize.}\hbox{\scriptsize.}}}}
\newcommand*{\vvvert}{|\mkern-2mu|\mkern-2mu|}
\newcommand*{\meas}[1]{\mu \mathopen{}\left(#1\right)\mathclose{}}
\newcommand*{\meass}[1]{\mu (#1)}
\newcommand*{\swidetilde}[1]{\smash{\widetilde{#1}}}
\renewcommand*{\theenumi}{(\alph{enumi})}

\newcommand*{\itoverline}[1]{\skew{3}{\overline}{#1}}
\newcommand*{\ittoverline}[1]{\skew{2}{\overline}{#1}}
\theoremstyle{plain}
\newtheorem{thm}{Theorem}[section]
\newtheorem{lem}[thm]{Lemma}
\newtheorem{pro}[thm]{Proposition}
\newtheorem{cor}[thm]{Corollary}
\theoremstyle{definition}
\newtheorem{df}[thm]{Definition}
\newtheorem{exa}[thm]{Example}
\newtheorem{rem}[thm]{Remark}
\numberwithin{equation}{section}
\begin{document}
%
%
%  ABSTRACT
%
%
\begin{abstract}
Newtonian spaces generalize first-order Sobolev spaces to abstract metric measure spaces. In this paper, we study regularity of Newtonian functions based on quasi-Banach function lattices. 
Their (weak) quasi-continuity is established, assuming density of continuous functions. The corresponding Sobolev capacity is shown to be an outer capacity. Assuming sufficiently high integrability of upper gradients, Newtonian functions are shown to be (essentially) bounded and (H\"older) continuous. Particular focus is put on the borderline case when the degree of integrability equals the ``dimension of the measure''. If Lipschitz functions are dense in a Newtonian space on a proper metric space, then locally Lipschitz functions are proven dense in the corresponding Newtonian space on open subsets, where no hypotheses (besides being open) are put on these sets.
\end{abstract}
%
%
%  TITLE
%
%
\title[Fine properties of Newtonian functions and the Sobolev capacity on metric spaces]{Fine properties of Newtonian functions and\\the Sobolev capacity on metric measure spaces}
\author{Luk\'{a}\v{s} Mal\'{y}}
\date{April 28, 2014}
\address{Department of Mathematics\\Link\"{o}ping University\\SE-581~83~Link\"{o}ping\\Sweden}
\address{Department of Mathematical Analysis\\Faculty of Mathematics and Physics\\Charles University in Prague\\Sokolovsk\'a 83\\CZ-186~75~Praha 8\\Czech Republic}
\email{lukas.maly@liu.se}
\maketitle{}
%
%
%  SECTION 1: INTRODUCTION
%
%
\section{Introduction}
\label{sec:intro}
First-order analysis in metric measure spaces requires a generalization of Sobolev spaces as the notion of a (distributional) gradient relies on the linear structure of $\Rbb^n$. The Newtonian approach makes use of the so-called upper gradients and weak upper gradients, which were originally introduced by Heinonen and Koskela~\cite{HeiKos0} and Koskela and MacManus~\cite{KosMac}, respectively. Shanmugalingam~\cite{ShaPhD,Sha} established the foundations for the Newtonian spaces $N^{1,p}$, based on the $L^p$ norm of a function and its (weak) upper gradient and hence corresponding to the classical Sobolev spaces $W^{1,p}$, cf.\@ Bj\"orn and Bj\"orn~\cite{BjoBjo} or Heinonen, Koskela, Shanmugalingam and Tyson~\cite{HeiKosShaTys}. Various authors have developed the elements of the Newtonian theory based on function norms other than $L^p$ in the past two decades, see e.g.\@~\cite{CosMir,Dur,HarHasPer,Tuo}. So far, foundations of the theory in utmost generality were obtained by Mal\'y in~\cite{Mal1,Mal2}, where complete quasi-normed lattices of measurable functions were considered as the base function spaces.
The question when Newtonian functions can be regularized using Lipschitz truncations lied in the focus of Mal\'y~\cite{Mal3}. The current paper goes further and studies regularity properties of Newtonian functions and of the corresponding Sobolev capacity.

One of the main points of interest is the so-called \emph{quasi-continuity}, which can be understood as a Luzin-type condition, where a set of arbitrarily small capacity can be found for each Newtonian function so that its restriction to the complement of that set is continuous. Existence of quasi-continuous representatives was first shown by Deny~\cite{Den} for functions of the unweighted Sobolev space $W^{1,2}(\Rbb^n, dx)$. An analogous result in Sobolev spaces $W^{1,p}(\Rbb^n, dx)$ is given in Federer and Ziemer~\cite{FedZie}, see also Mal\'y and Ziemer~\cite{MalZie}, and the situation in weighted Sobolev spaces is discussed in Heinonen, Kilpel\"ainen, Martio~\cite{HeiKilMar}. In metric spaces, Shanmugalingam~\cite{Sha} showed that Newtonian functions in $N^{1,p}$ have quasi-continuous representatives if the metric space is endowed with a doubling measure and supports a $p$-Poincar\'e inequality (see Definition~\ref{df:PI} below). The hypotheses were weakened in Bj\"orn, Bj\"orn and Shanmugalingam~\cite{BjoBjoSha}, where density of continuous functions was proven sufficient to obtain existence of quasi-continuous representatives of $N^{1,p}$ functions. The current paper provides an analogous result for the Newtonian space $\NX$ built upon an arbitrary quasi-Banach function lattice $X$.

Furthermore in~\cite{BjoBjoSha}, all $N^{1,p}$ functions were proven to be quasi-continuous given that the metric space is \emph{proper} (i.e., if all bounded closed sets are compact). In order to show a similar property of all $\NX$ functions in proper metric spaces, the quasi-Banach function lattice $X$ needs to possess the \emph{Vitali--Carath\'eodory property} (i.e., the quasi-norm of a function can be approximated by the quasi-norms of its lower semicontinuous majorants). Bj\"orn, Bj\"orn and Mal\'y~\cite{BjoBjoMal} give counterexamples that show that this property is vital.

Since the Vitali--Carath\'eodory property is crucial for the presented results, we will look into the question when a general quasi-Banach function lattice $X$ possesses it. Vitali~\cite{Vit} proved that the $L^1$ norm of a measurable function on $\Rbb^n$ can be approximated by the $L^1$ norms of its lower semicontinuous majorants. His result can be easily generalized to $L^p(\Rbb^n)$ with $0<p<\infty$. We will show that it suffices that $X$ contains all simple functions (with support of finite measure) and these have absolutely continuous quasi-norm. Moreover, counterexamples are provided when~$X$ violates either of these two conditions.

Quasi-continuity of Newtonian functions in $\NX$ is closely connected with regularity of the Sobolev capacity $C_X$. Namely, under the assumption that continuous functions are dense in the Newtonian space $\NX$, all Newtonian functions are quasi-continuous if and only if $C_X$ (or an equivalent capacity in case $X$ is merely quasi-normed) is an outer capacity. Actually, the density of continuous functions need not be assumed in the forward implication.

Furthermore, quasi-continuity can be applied to show that locally Lipschitz functions are dense in a Newtonian space on any open subset of a metric space provided that locally Lipschitz functions are dense in the Newtonian space on the entire metric space. The noteworthy part of this claim is that the open subset as a metric subspace need not support any Poincar\'e inequality and the restriction of the measure need not be doubling any more. In general, it is however impossible to obtain density of Lipschitz functions.

It was observed already by Morrey~\cite{Mor} in 1940 that the classical Sobolev functions in $\Rbb^n$ have (H\"older) continuous representatives if the degree of summability of the weak gradients is sufficiently high compared to the dimension. A similar result based on a $p$-Poincar\'e inequality was obtained by Haj\l{}asz and Koskela~\cite{HajKos} in metric spaces endowed with a doubling measure, after introducing an analogue of the dimension. As we are considering Newtonian spaces based on general function lattices, our tools suffice to study the borderline case when the degree of summability (in terms of a Banach function lattice quasi-norm) of upper gradients is essentially equal to the ``dimension'' of a doubling measure. We will establish conditions that guarantee that all Newtonian functions are essentially bounded and have continuous representatives (in equivalence classes given by equality up to sets of capacity zero). If the metric measure space is in addition locally compact, then all Newtonian functions are in fact continuous.

The structure of the paper is the following. Section~\ref{sec:prelim} provides an overview of the used notation and preliminaries in the area of quasi-Banach function lattices and Newtonian spaces. In Section~\ref{sec:capacity}, we study the Sobolev capacity, still without the assumption on density of continuous functions. After that, in Section~\ref{sec:quasicontinuity}, we move on to quasi-continuity and its consequences for the Sobolev capacity and continuity of Newtonian functions. Density of locally Lipschitz functions on general open sets is shown in Section~\ref{sec:density}. A very short introduction to rearrangement-invariant spaces is provided in Section~\ref{sec:boundedness}, whose main focus however lies in establishing sufficient conditions for Newtonian functions to be essentially bounded. In Section~\ref{sec:continuity}, existence of continuous representatives and continuity of all representatives is discussed. Several lemmata for calculus of (minimal) weak upper gradients are given in the appendix.
%
%
%  SECTION 2: PRELIMINARIES
%
%
\section{Preliminaries}
\label{sec:prelim}
We assume throughout the paper that $\Pcal = (\Pcal, \dd, \mu)$ is a metric measure space equipped with a metric $\dd$ and a $\sigma$-finite Borel regular measure $\mu$ such that every ball in $\Pcal$ has finite positive measure. In our context, Borel regularity means that all Borel sets in $\Pcal$ are $\mu$-measurable and for each $\mu$-measurable set $A$ there is a Borel set $D\supset A$ such that $\meas{D} = \meas{A}$. Since $\mu$ is Borel regular and $\Pcal$ can be decomposed into countably many (possibly overlapping) open sets of finite measure, it is outer regular, see Mattila~\cite[Theorem 1.10]{Mat}.

The open ball centered at $x\in \Pcal$ with radius $r>0$ will be denoted by $B(x,r)$. Given a ball $B=B(x,r)$ and a scalar $\lambda > 0$, we let $\lambda B = B(x,\lambda r)$. We say that $\mu$ is a \emph{doubling} measure, if there is a constant $c_{\dbl}\ge1$ such that $\meas{2B} \le c_{\dbl} \meas{B}$ for every ball $B$. Note that we will assume that $\mu$ satisfies the doubling condition only in Sections~\ref{sec:boundedness} and~\ref{sec:continuity}, where essential boundedness and continuity of Newtonian functions are studied.

A metric space is \emph{proper} if all its closed and bounded subsets are compact. A doubling metric measure space (and hence a metric space with a doubling measure) is proper if and only if it is complete, see Bj\"orn and Bj\"orn~\cite[Proposition 3.1]{BjoBjo}.

Let $\Mcal(\Pcal, \mu)$ denote the set of all extended real-valued $\mu$-measurable functions on~$\Pcal$. The symbol $\Lipc(\Omega)$ stands for Lipschitz continuous functions with compact support in $\Omega$. The set of extended real numbers, i.e., $\Rbb \cup \{\pm \infty\}$, will be denoted by $\overline{\Rbb}$. We will also use $\Rbb^+$, which denotes the set of positive real numbers, i.e., the interval $(0, \infty)$. The symbol $\Nbb$ will denote the set of positive integers, i.e., $\{1,2, \ldots\}$. We define the \emph{integral mean} of a measurable function $u$ over a set $E$ of finite positive measure as
\[
  u_E \coloneq \fint_E u\,d\mu = \frac{1}{\mu(E)} \int_E u\,d\mu,
\]
whenever the integral on the right-hand side exists, not necessarily finite though.
We write $E \Subset A$ if $\itoverline{E}$ is a compact subset of $A$. The notation $L \lesssim R$ will be used to express that there exists a constant $c>0$, perhaps dependent on other constants within the context, such that $ L \le cR$. If $L \lesssim R$ and simultaneously $R \lesssim L$, then we will simply write $L \approx R$ and say that the quantities $L$ and $R$ are \emph{comparable}. The words \emph{increasing} and \emph{decreasing} will be used in their non-strict sense.

A linear space $X = X(\Pcal, \mu)$ of equivalence classes of functions in $\Mcal(\Pcal, \mu)$ is a \emph{quasi-Banach function lattice} over $(\Pcal, \mu)$ equipped with the quasi-norm $\|\cdot\|_X$ if the following axioms hold:
\begin{enumerate}
  \renewcommand{\theenumi}{(P\arabic{enumi})}
  \setcounter{enumi}{-1}
  \item \label{df:qBFL.initial} $\|\cdot\|_X$ determines the set $X$, i.e., $X = \{u\in \Mcal(\Pcal, \mu)\colon \|u\|_X < \infty\}$;
  \item \label{df:qBFL.quasinorm} $\|\cdot\|_X$ is a \emph{quasi-norm}, i.e., 
  \begin{itemize}
    \item $\|u\|_X = 0$ if and only if $u=0$ a.e.,
    \item $\|au\|_X = |a|\,\|u\|_X$ for every $a\in\Rbb$ and $u\in\Mcal(\Pcal, \mu)$,
    \item there is a constant $\cconc \ge 1$, the so-called \emph{modulus of concavity}, such that $\|u+v\|_X \le \cconc(\|u\|_X+\|v\|_X)$ for all $u,v \in \Mcal(\Pcal, \mu)$;
  \end{itemize}
  \item $\|\cdot\|_X$ satisfies the \emph{lattice property}, i.e., if $|u|\le|v|$ a.e., then $\|u\|_X\le\|v\|_X$;
    \label{df:BFL.latticeprop}
  \renewcommand{\theenumi}{(RF)}
  \item \label{df:qBFL.RF} $\|\cdot\|_X$ satisfies the \emph{Riesz--Fischer property}, i.e., if $u_n\ge 0$ a.e.\@ for all $n\in\Nbb$, then $\bigl\|\sum_{n=1}^\infty u_n \bigr\|_X \le \sum_{n=1}^\infty \cconc^n \|u_n\|_X$, where $\cconc\ge 1$ is the modulus of concavity. Note that the function $\sum_{n=1}^\infty u_n$ needs to be understood as a pointwise (a.e.\@) sum.
\end{enumerate}
Observe that $X$ contains only functions that are finite a.e., which follows from \ref{df:qBFL.quasinorm} and \ref{df:BFL.latticeprop}. In other words, if $\|u\|_X<\infty$, then $|u|<\infty$ a.e.

Throughout the paper, we will also assume that the quasi-norm $\|\cdot\|_X$ is \emph{continuous}, i.e., if $\|u_n - u\|_X \to 0$ as $n\to\infty$, then $\|u_n\|_X \to \|u\|_X$. We do not lose any generality by this assumption as the Aoki--Rolewicz theorem (see Benyamini and Lindenstrauss~\cite[Proposition H.2]{BenLin} or Maligranda~\cite[Theorem~1.2]{Mali}) implies that there is always an equivalent quasi-norm that is an \emph{$r$-norm}, i.e., it satisfies
\[
  \|u + v\|^r \le \|u\|^r + \|v\|^r,
\]
where $r = 1/(1+ \log_2 \cconc) \in (0, 1]$, which implies the continuity. The theorem's proof shows that such an equivalent quasi-norm retains the lattice property. Moreover, $\|\cdot\|^r$ satisfies \ref{df:qBFL.RF} without any constants, i.e., $\bigl\|\sum_{n=1}^\infty u_n \bigr\|^r \le \sum_{n=1}^\infty \|u_n\|^r$. 

It is worth noting that the Riesz--Fischer property is actually equivalent to the completeness of the quasi-normed space $X$, given that the conditions \ref{df:qBFL.initial}--\ref{df:BFL.latticeprop} are satisfied and that the quasi-norm is continuous, see Maligranda~\cite[Theorem 1.1]{Mali}.

If $\cconc = 1$, then the functional $\| \cdot \|_X$ is a norm. We then drop the prefix \emph{quasi} and hence call $X$ a \emph{Banach function lattice}.

A (quasi)Banach function lattice $X = X(\Pcal, \mu)$ is a \emph{(quasi)Banach function space} over $(\Pcal, \mu)$ if the following axioms are satisfied as well:
\begin{enumerate}
  \renewcommand{\theenumi}{(P\arabic{enumi})}
  \setcounter{enumi}{2}
  \item $\|\cdot\|_X$ satisfies the \emph{Fatou property}, i.e., if $0\le u_n \upto u$ a.e., then $\|u_n\|_X\upto\|u\|_X$;
  \item \label{df:BFS.finmeasfinnorm} if a measurable set $E \subset \Pcal$ has finite measure, then $\|\chi_E\|_X < \infty$;
  {
  \item for every measurable set $E\subset \Pcal$ with $\meas{E}<\infty$ there is $C_E>0$ such that $\int_E |u|\,d\mu \le C_E \|u\|_X$ for every measurable function $u$.
    \label{df:BFL.locL1}
  } 
\end{enumerate}
Note that the Fatou property implies the Riesz--Fischer property. Axiom \ref{df:BFS.finmeasfinnorm} is equivalent to the condition that $X$ contains all simple functions (with support of finite measure). Due to the lattice property, \ref{df:BFS.finmeasfinnorm} can be also equivalently characterized as embedding of $L^\infty(\Pcal, \mu)$ into $X$ on sets of finite measure. Finally, condition~\ref{df:BFL.locL1} describes that $X$ is embedded into $L^1(\Pcal, \mu)$ on sets of finite measure.

In the further text, we will slightly deviate from this rather usual definition of (quasi)\allowbreak{}Banach function lattices and spaces. Namely, we will consider $X$ to be a linear space of functions defined everywhere instead of equivalence classes defined a.e. Then, the functional $\|\cdot\|_X$ is really only a (quasi)seminorm. Unless explicitly stated otherwise, we will always assume that $X$ is a quasi-Banach function lattice.

We will say that $X$ is \emph{continuously embedded} in $Y_\loc$, denoted by $X \emb Y_\loc$, if for every ball $B \subset \Pcal$ there is $\cemb(B) > 0$ such that $\|u \chi_B\|_Y \le \cemb(B) \|u \chi_B\|_X$ whenever $u\in \Mcal(\Pcal, \mu)$. The global continuous embedding $X\emb Y$ is defined in a similar fashion by letting $B=\Pcal$.

A function $u\in X$ has \emph{absolutely continuous quasi-norm in $X$}, if it satisfies that
\begin{enumerate}
  \renewcommand{\theenumi}{(AC)}
  \item \label{df:AC}
  $\| u \chi_{E_n} \|_X \to 0$ as $n\to\infty$ whenever $\{E_n\}_{n=1}^\infty$ is a decreasing sequence of measurable sets with $\meas{\bigcap_{n=1}^\infty E_n} = 0$.
\end{enumerate}
The quasi-norm $\| \cdot \|_X$ is \emph{absolutely continuous} if every $u \in X$ has absolutely continuous quasi-norm in $X$.

It follows from the dominated convergence theorem that the $L^p$ norm is absolutely continuous for $p\in (0, \infty)$. On the other hand, $L^\infty$ lacks this property apart from in a few exceptional cases. For example, if $\mu$ is atomic, $0<\delta\le \meas{A}$ for every atom $A\subset \Pcal$, and $\meas{\Pcal}<\infty$, then every quasi-Banach function lattice has absolutely continuous quasi-norm since the condition $\meas{\bigcap_{n=1}^\infty E_n} = 0$ implies that $E_n=\emptyset$ for all sufficiently large $n\in\Nbb$. However, atomic measures lie outside of the main scope of our interest.

In a quasi-Banach function lattice $X$ whose simple functions have absolutely continuous quasi-norm in $X$ (similarly as in the setting of Lebesgue spaces $L^p$ with $p<\infty$), we may approximate the quasi-norm of a function by the quasi-norms of its lower semicontinuous (lsc) majorants.
%
% ---------------------------------------------------------
%
\begin{pro}[Vitali--Carath\'eodory theorem] 
\label{pro:VitaliCarath}
Let $X$ be a quasi-Banach function lattice. Suppose that $\chi_B \in X$ and it satisfies \ref{df:AC} whenever $B\subset \Pcal$ is bounded. If $u:\Pcal \to \overline{\Rbb}$ is measurable, then
\begin{equation}
  \label{eq:VitaliCaratheodory}
  \| u \|_X = \inf\{ \|v\|_X: v\ge |u| \mbox{ on $\Pcal$ and }v\in \lsc(\Pcal)\}.
\end{equation}
\end{pro}
%
% ---------------------------------------------------------
%
In particular, the hypotheses are fulfilled if $X$ is a rearrangement-invariant quasi-Banach function space (see the definition in Section~\ref{sec:boundedness} below) whose fundamental function $\phi_X$ satisfies $\lim_{t\to0\limplus} \phi_X(t) = 0$, which can be also expressed as $X \nsubset L^\infty$.

The theorem's origin can be dated back to 1905, when Vitali~\cite{Vit} showed that every function $f \in L^1(\Rbb^n, d\mu)$ coincides a.e.\@ with a function of the Baire class 2. Namely, he showed that there exist sequences $\{u_k\}_{k=1}^\infty$ and $\{l_k\}_{k=1}^\infty$ of upper semicontinuous (usc) minorants and lsc majorants, respectively, such that $u_k\upto f$ and $l_k \downto f$ everywhere in $\Rbb^n$, and $\|u_k\|_{L^1} \to \|f\|_{L^1}$ and $\|l_k\|_{L^1} \to \|f\|_{L^1}$ as $k\to\infty$. In 1918, Carath\'eodory~\cite{Car} has shown that usc minorants and lsc majorants with the same convergence properties exist for every $f\in \Mcal(\Rbb^n, d\mu)$, i.e., even if $f \notin L^1(\Rbb^n, d\mu)$.

The lsc majorants in the proof below are constructed similarly as in E.~\&~M.~J\"arvenp\"a\"a, K.~\&~S.~Rogovin, and Shanmugalingam~\cite[Lemma~2.3]{Jar2Rog2Sha}, where merely $X=L^p$ with $p \in [1, \infty)$ was considered and $u$ was assumed Borel measurable.

Note also that both the absolute continuity and \ref{df:BFS.finmeasfinnorm} are vital for the proposition. For example, take $u = \chi_{\{0\}}$ on $\Rbb$. Then, $\|u\|_X=0$ for every quasi-Banach function lattice $X$ as $u=0$ a.e. If $X = L^\infty(\Rbb)$, which lacks the absolute continuity, then $\|v\|_X \ge 1$ for every lsc majorant $v$ of $u$. The norm $\|f\|_Y = \int_\Rbb |f(t)/t|\,dt$ gives rise to a function space that fails to contain $\chi_B$ for all bounded sets $B\subset\Rbb$. Then, $\|v\|_Y = \infty$ for every lsc majorant $v$ of $u$ since $v>1/2$ in some open neighborhood of zero.
%
% ---------------------------------------------------------
%
\begin{proof}[Proof of Proposition~\ref{pro:VitaliCarath}]
Let $u$ be given. If $\|u\|_X = \infty$, then the desired identity holds trivially. Suppose instead that $\|u\|_X<\infty$. Without loss of generality, we may assume that $u$ is non-negative. For an arbitrary $\eps > 0$, we want to find $v\in \lsc(\Pcal)$ such that $v\ge u$ and $\|v-u\|_X < \eps$. Fixing an arbitrary point $x_0\in\Pcal$, we can decompose $\Pcal$ as a union of an open ball and open spherical shells centered in $x_0$, i.e., $\Pcal = \bigcup_{k=1}^\infty \Pcal_k$, where $\Pcal_k = \{ x\in \Pcal: k-2< \dd(x,x_0) < k\}$.

Let $\eps>0$ be fixed. For each $k\in\Nbb$, we will find an lsc function $v_k$ that dominates $u$ on $\Pcal_k$ while $\|(v_k - u)\chi_{\Pcal_k}\|_X < \eps / (2\cconc)^k$, where $\cconc\ge 1$ is the modulus of concavity of $X$ (i.e., the constant in the triangle inequality in \ref{df:qBFL.quasinorm}).

Fix $k\in\Nbb$ and let $\delta = \eps / \cconc^2 (2\cconc)^k (2+\|\chi_{\Pcal_k}\|_X)$. Let $E_\infty=\Pcal_k \cap u^{-1}(\infty)$. Then, $\meas{E_\infty} = 0$. For every $n\in\Nbb$, we define $E_n = \Pcal_k \cap u^{-1}([(n-1)\delta, n\delta))$. Due to the outer regularity of $\mu$ and the absolute continuity of the (quasi)norm of $\chi_{\Pcal_k}$ in $X$, there are open sets $U_n$ and $V_n$ such that $E_n \subset U_n \subset \Pcal_k$ with $\|\chi_{U_n \setminus E_n}\|_X < 1/n(2\cconc)^n$, and $E_\infty \subset V_n \subset \Pcal_k$ with $\|\chi_{V_n}\|_X < \delta/(2\cconc)^n$. Define now $v_k: \Pcal \to [0, \infty]$ by
\[
  v_k(x) = \sum_{n=1}^\infty (n\delta \chi_{U_n} + \chi_{V_n}).
\]
Obviously, $v_k\in \lsc(\Pcal_k)$ and $v_k\ge u$ on $\Pcal_k$. Then, we can estimate
\begin{align*}
  v_k(x) - u(x) & \le \sum_{n=1}^\infty \bigl(n\delta \chi_{U_n}(x) + \chi_{V_n}(x) - (n-1)\delta\chi_{E_n}(x)\bigr) \\
   & \le \sum_{n=1}^\infty \bigl(n\delta \chi_{U_n \setminus E_n}(x) + \delta \chi_{E_n}(x) +  \chi_{V_n}(x)\bigr) \\
   & = \sum_{n=1}^\infty n\delta \chi_{U_n \setminus E_n}(x) + \delta \chi_{\Pcal_k}(x) +  \sum_{n=1}^\infty \chi_{V_n}(x).
\end{align*}
The triangle inequality and the Riesz--Fischer property give that
\begin{align}
  \notag
  \|(v_k - u)\chi_{\Pcal_k}\|_X &\le \cconc^2 \biggl( \delta \sum_{n=1}^\infty n \cconc^n \|\chi_{U_n \setminus E_n}\|_X + \delta \|\chi_{\Pcal_k} \|_X + \sum_{n=1}^\infty \cconc^n \|\chi_{V_n}\|_X \biggr)
\\
& < \delta \cconc^2 (2+\|\chi_{\Pcal_k}\|_X) = \frac{\eps}{(2\cconc)^k}\,.
  \label{eq:VitCar-vk}
\end{align}

Let now $v(x) = \max_{k\in\Nbb} v_k(x)$ for all $x\in\Pcal$. Note that for each $x\in\Pcal$, there may be at most two values of $k$ such that $v_k(x)>0$ and in that case they are consecutive. Define thus $\Pcal_k' = \{ x \in \Pcal_k: v(x) = v_k(x) > v_{k+1}(x)\}$. Then, the sets $\Pcal_k'$, $k\in\Nbb$, are pairwise disjoint and $v(x) = u(x) = 0$ whenever $x\in \Pcal \setminus \bigcup_{k=1}^\infty \Pcal_k'$, which together with \eqref{eq:VitCar-vk} leads to the estimate
\[
  \| v - u \|_X = \biggl\| \sum_{k=1}^\infty (v-u)\chi_{\Pcal_k'} \biggr\|_X = \biggl\| \sum_{k=1}^\infty (v_k-u)\chi_{\Pcal_k'} \biggr\|_X \le \sum_{k=1}^\infty \cconc^k \|(v_k-u)\chi_{\Pcal_k}\|_X<\eps.
\qedhere
\]
\end{proof}
%
% ---------------------------------------------------------
%
By a \emph{curve} in $\Pcal$ we will mean a non-constant continuous mapping $\gamma: I\to \Pcal$ with finite total variation (i.e., length of $\gamma(I)$), where $I \subset \Rbb$ is a compact interval. Thus, a curve can be (and we will always assume that all curves are) parametrized by arc length $ds$, see e.g.\@ Heinonen~\cite[Section~7.1]{Hei}. Note that every curve is Lipschitz continuous with respect to its arc length parametrization. The family of all non-constant rectifiable curves in $\Pcal$ will be denoted by $\Gamma(\Pcal)$. By abuse of notation, the image of a curve $\gamma$ will also be denoted by $\gamma$. 

A statement holds for \emph{$\Mod_X$-a.e.\@} curve if the family of exceptional curves $\Gamma_e$, for which the statement fails, has \emph{zero $X$-modulus}, i.e., if there is a Borel function $\rho \in X$ such that $\int_\gamma \rho\,ds = \infty$ for every curve $\gamma \in \Gamma_e$ (see~\cite[Proposition 4.8]{Mal1}).

\begin{df}
\label{df:ug}
  Let $u: \Pcal \to \overline{\Rbb}$. Then, a Borel function $g: \Pcal \to [0, \infty]$ is an \emph{upper gradient} of $u$ if
\begin{equation}
 \label{eq:ug_def}
 |u(\gamma(0)) - u(\gamma(l_\gamma))| \le \int_\gamma g\,ds
\end{equation}
for every curve $\gamma: [0, l_\gamma]\to\Pcal$. To make the notation easier, we are using the convention that $|(\pm\infty)-(\pm\infty)|=\infty$. If we allow $g$ to be a measurable function and~\eqref{eq:ug_def} to hold only for $\Mod_X$-a.e.\@ curve $\gamma: [0, l_\gamma]\to\Pcal$, then $g$ is an \emph{$X$-weak upper gradient}.
\end{df}

Observe that the ($X$-weak) upper gradients are by no means given uniquely. Indeed, if we have a function $u$ with an ($X$-weak) upper gradient $g$, then $g+h$ is another ($X$-weak) upper gradient of $u$ whenever $h\ge0$ is a Borel (measurable) function.
%
% ---------------------------------------------------------
%
\begin{df}
We say that function $u\in \Mcal(\Pcal, \mu)$ belongs to the \emph{Dirichlet space} $DX$ if it has an upper gradient $g \in X$. Let
\begin{equation}
  \label{eq:def-N1X-norm}
  \|u\|_{\NX} = \| u \|_X + \inf_g \|g\|_X,
\end{equation}
where the infimum is taken over all upper gradients $g$ of $u$.
The \emph{Newtonian space} based on $X$ is the space
\[
  \NX = \NX (\Pcal, \mu) \coloneq \{u\in\Mcal(\Pcal, \mu): \|u\|_{\NX} <\infty \} = X \cap DX.
\]
Given a measurable set $\Theta \subset \Pcal$, we define
\[
  \NnX(\Theta) = \{u|_\Theta: u\in\NX\textup{ and } u=0\textup{ in }\Pcal \setminus \Theta\}.
\]
\end{df}
Note that we may define $DX$ via $X$-weak upper gradients and take the infimum over all $X$-weak upper gradients $g$ of $u$ in \eqref{eq:def-N1X-norm} without changing the value of the Newtonian quasi-norm, see~\cite[Corollary~5.7]{Mal1}. Let us also point out that we assume that functions are defined everywhere, and not just up to equivalence classes \mbox{$\mu$-}almost everywhere. This is essential for the notion of upper gradients since they are defined by a pointwise inequality.

It has been shown in~\cite{Mal2} that the infimum in \eqref{eq:def-N1X-norm} is attained for functions in $\NX$ by a \emph{minimal $X$-weak upper gradient}. Such an $X$-weak upper gradient is minimal both normwise and pointwise (a.e.\@) among all ($X$-weak) upper gradients in $X$, whence it is given uniquely up to equality a.e. The minimal $X$-weak upper gradient of a function $u\in\NX$ will be denoted by $g_u \in X$.

The functional $\|\cdot\|_\NX$ is a quasi-seminorm on $\NX$ and a quasi-norm on $\NtX \coloneq \NX/\mathord\sim$, where the equivalence relation $u\sim v$ is given by $\|u-v\|_\NX = 0$. The modulus of concavity for $\NX$ (i.e., the constant in the triangle inequality) is equal to $\cconc$, the modulus of concavity for $X$. Furthermore, the Newtonian space $\NtX$ is complete and thus a quasi-Banach space, see~\cite[Theorem~7.1]{Mal1}.
%
%
%  SECTION 3: SOBOLEV CAPACITY
%
%
\section{Sobolev capacity}
\label{sec:capacity}
When working with first-order analysis, it is the Sobolev capacity that provides a set function that distinguishes which sets do not carry any information about a Newtonian function and thus are negligible. In this section, we will show a certain rigidity property of the capacity. Furthermore, if $\Pcal$ is locally compact and if the Vitali--Carath\'eodory theorem holds, then we will obtain that the capacity is outer regular on sets of zero capacity.
\begin{df}
\label{df:capacity}
The \emph{(Sobolev) $X$-capacity} of a set $E\subset \Pcal$ is defined as
\[
  C_X(E) = \inf\{ \|u\|_{\NX}: u\ge 1 \mbox{ on }E\}.
\]
If $X$ is $r$-normed for some $r \in (0, \infty)$, then we define the \emph{(Sobolev) $X,\!r$-capacity} of a set $E\subset \Pcal$ by
\[
  \widetilde{C}_{X,r}(E) = \inf\{ (\|u\|^r_{X} + \|g\|^r_X)^{1/r}: u\ge 1 \mbox{ on $E$ and $g$ is an upper gradient of $u$}\}.
\]
The function lattice $X$ will be implicitly assumed to be $r$-normed whenever the capacity $\widetilde{C}_{X,r}$ is used.
We say that a property of points in $\Pcal$ holds \emph{$C_X$-quasi-everywhere ($C_X$-q.e.)} if the set of exceptional points has $X$-capacity zero. 
\end{df}
Observe that the capacities $C_X$ and $\widetilde{C}_{X,r}$ are equivalent, viz.,
\[
  \min\{1, 2^{1/r-1}\} C_X(E) \le \widetilde{C}_{X,r}(E) \le \max\{1, 2^{1/r-1}\} C_X(E) \quad\mbox{for every $E\subset \Pcal$.}
\]
Therefore, it is of no importance whether the notion of quasi-everywhere is defined using $C_X$ or $\widetilde{C}_{X,r}$. The capacities $C_X(E)$ and $\widetilde{C}_{X,r}(E)$ may be equivalently defined considering only functions $u$ such that $\chi_E\le u \le 1$, see~\cite[Proposition~3.2]{Mal1}. If $X$ is normed, then $C_X = \widetilde{C}_{X,1}$. Despite the dependence on $X$, we will often write simply \emph{capacity} and \emph{q.e.\@} whenever there is no risk of confusion of the base function space.

A capacity $C$ is an \emph{outer capacity}, if $C(E) = \inf_{G} C(G)$, where the infimum is taken over all open sets $G \supset E$. Based on the quasi-continuity of Newtonian functions, we will show in Proposition~\ref{pro:C_X-outer} below that $\widetilde{C}_{X,r}$ is an outer capacity.

It was established in~\cite[Theorem 3.4]{Mal1} that $C_X$ is an outer measure on $\Pcal$ if $\cconc=1$. Otherwise, the set function $C_X$ is merely $\sigma$-quasi-additive, i.e.,
\[
  C_X \biggl( \bigcup_{j=1}^\infty E_j \biggr) \le \sum_{j=1}^\infty \cconc^j C_X(E_j).
\]
On the contrary, $\widetilde{C}_{X,r}(\cdot)^r$ is always an outer measure on $\Pcal$, as is shown next.
\begin{lem}
The function $\widetilde{C}_{X,r}(\cdot)^r$ is $\sigma$-subadditive, i.e., 
\[
 \widetilde{C}_{X,r} \biggl( \bigcup_{j=1}^\infty E_j \biggr)^r \le \sum_{j=1}^\infty \widetilde{C}_{X,r}(E_j)^r
\]
whenever $E_1, E_2, \ldots \subset \Pcal$.
\end{lem}
\begin{proof}
If $\widetilde{C}_{X,r}(E_j) = \infty$ for some $j\in\Nbb$, then the wanted inequality holds trivially. Suppose therefore that $\widetilde{C}_{X,r}(E_j) < \infty$ for every $j\in\Nbb$. For each $E_j$, $j\in\Nbb$, we can hence find $u_j\in \NX$ with an upper gradient $g_j\in X$ such that $\chi_{E_j} \le u_j \le 1$, and $\|u_j\|_X^r + \|g_j\|_X^r < \widetilde{C}_{X,r}(E_j)^r + 2^{-j} \eps$. Let $u = \sup_{j\ge1} u_j$ and $g = \sup_{j\ge1} g_j$. Then, $\chi_{\bigcup_{j=1}^\infty E_j} \le u \le 1$, while $g$ is an upper gradient of $u$ by~\cite[Lemma 3.3]{Mal1}. Hence,
\begin{align*}
  \widetilde{C}_{X,r}\biggl( \bigcup_{j=1}^\infty E_j \biggr)^r & \le \|u\|_X^r +\|g\|_X^r = \Bigl\| \sup_{j\ge 1} u_j\Bigr\|_X^r + \Bigl\| \sup_{j\ge 1} g_j\Bigr\|_X^r \le \biggl\| \sum_{j=1}^\infty u_j\biggr\|_X^r + \biggl\| \sum_{j=1}^\infty g_j\biggr\|_X^r \\
  & \le \sum_{j=1}^\infty (\|u_j\|_X^r + \|g_j\|_X^r) < \sum_{j=1}^\infty \Bigl(  \widetilde{C}_{X,r}(E_j)^r + \frac{ \eps}{2^j}\Bigr) = \eps + \sum_{j=1}^\infty  \widetilde{C}_{X,r}(E_j)^r\,.
\end{align*}
Letting $\eps\to 0$ completes the proof.
\end{proof}

If $C_X(E) = 0$, then $\mu(E) = 0$. The converse is however not true in general. The natural equivalence classes in $\NX$, where $u$ and $v$ are equivalent if $\|u-v\|_\NX = 0$, are in fact given by equality q.e.\@ as shown in~\cite[Corollary 6.16]{Mal1}.

The following proposition shows a certain rigidity property of the Sobolev capacity of an open set, which is an important hypothesis in Proposition~\ref{pro:qc_ae=qe} below. The usual Sobolev capacity in $\Rbb^n$ has this property trivially by definition, see Heinonen, Kilpel\"ainen and Martio~\cite[Definition~2.35]{HeiKilMar}. The idea of the claim and its proof originates in~\cite[Proposition~5.22]{BjoBjo}. It is noteworthy that we do not need that $C_X$ (or $\widetilde{C}_{X,r}$) is an outer capacity to obtain this result.
\begin{pro}
\label{pro:cap_rigid}
Let $G \subset \Pcal$ be open and suppose that $\meas{E} = 0$. Then,
\[
 C_X(G) = C_X(G\setminus E)\quad\mbox{and}\quad\widetilde{C}_{X,r}(G) = \widetilde{C}_{X,r}(G\setminus E).
\]
\end{pro}
\begin{proof}
Obviously, $C_X(G) \ge C_X(G\setminus E)$ as the capacity is monotone.

The converse inequality holds trivially if $C_X(G\setminus E) = \infty$. Hence, suppose that $C_X(G\setminus E) < \infty$. Let $\eps>0$. Then, there is $u \in \NX$ with an upper gradient $g\in X$ such that $\chi_{G\setminus E} \le u \le 1$ and $\|u\|_\NX \le \|u\|_X + \|g\|_X < C_X(G\setminus E) + \eps$.

Let $v = \max\{\chi_G, u\}$. Then, $u=v$ outside of $G \cap E$, whose measure is zero. Hence, $\|u\|_X = \|v\|_X$. We will show that $g$ is an $X$-weak upper gradient of $v$. Let $\gamma: [0, l_\gamma] \to \Pcal$ be a curve such that $\lambda^1(\gamma^{-1}(G \cap E)) = 0$ and \eqref{df:ug} is satisfied for all subcurves $\gamma' = \gamma|_I$, where $I \subset [0, l_\gamma]$ is a closed interval. By~\cite[Lemma~4.9 and Corollary~5.9]{Mal1}, $\Mod_X$-a.e.\@ curve $\gamma$ satisfies these conditions.

If $\gamma(0) \in G\cap E$, then there is $\alpha \in (0, l_\gamma)$ such that $\gamma(\alpha) \in G \setminus E$ as $\gamma^{-1}(G)$ is open in $[0, l_\gamma]$ and $\lambda^1(\gamma^{-1}(G \cap E)) = 0$. If $\gamma(0) \notin G\cap E$, then we set $\alpha = 0$. We obtain that $v(\gamma(0)) = u(\gamma(\alpha))$. 

Similarly, if $\gamma(l_\gamma) \in G\cap E$, then there is $\beta \in (\alpha, l_\gamma)$ such that $\gamma(\beta) \in G \setminus E$. We set $\beta = l_\gamma$ otherwise. Consequently, $v(\gamma(l_\gamma)) = u(\gamma(\beta))$. Therefore,
\begin{equation}
  \label{eq:cap_rigid_wug}
  |v(\gamma(0)) - v(\gamma(l_\gamma))| = |u(\gamma(\alpha)) - u(\gamma(\beta))| \le \int_{\gamma|_{[\alpha, \beta]}} g\,ds \le \int_\gamma g\,ds.
\end{equation}
Thus, $g$ is an $X$-weak upper gradient of $v$ as \eqref{eq:cap_rigid_wug} holds for $\Mod_X$-a.e.\@ curve $\gamma$. Hence
\[
  C_X(G) \le \|v\|_\NX \le \|v\|_X + \|g\|_X = \|u\|_X + \|g\|_X < C_X(G\setminus E) + \eps.
\]
Letting $\eps \to 0$, we see that $C_X(G) \le C_X(G\setminus E)$ as needed.

The equality for $\widetilde{C}_{X,r}$ can be shown analogously.
\end{proof}
%
% ---------------------------------------------------------
%
For an arbitrary set, adding a zero set (with respect to the capacity) does not change the capacity of the set even if $C_X$ is not subadditive but merely quasi-additive as we are now about to see.
\begin{lem}
\label{lem:C_X-addnull}
Let $E,F \subset \Pcal$. Suppose that $C_X(F) = 0$. Then, $C_X(E\cup F) = C_X(E)$ and $\widetilde{C}_{X,r}(E\cup F) = \widetilde{C}_{X,r}(E)$.
\end{lem}
%
% ---------------------------------------------------------
%
\begin{proof}
By monotonicity, $C_X(E\cup F) \ge C_X(E)$ and $\widetilde{C}_{X,r}(E\cup F) \ge \widetilde{C}_{X,r}(E)$. The converse inequality for $\widetilde{C}_{X,r}$ follows from the $\sigma$-subadditivity of $\widetilde{C}^r_{X,r}$. 

If $C_X(E) = \infty$, then the converse inequality holds trivially. 
Suppose now that $C_X(E)<\infty$. Let $\eps>0$ and let $\chi_E \le u\in \NX$ be such that $\|u\|_\NX \le \|u\|_X + \|g\|_X < C_X(E) + \eps$, where $g\in X$ is an upper gradient of $u$. Set $v= u + \chi_F$. Then, $v=u$ q.e., whence $\|v\|_X = \|u\|_X$ and $g$ is a weak upper gradient of $v$ by~\cite[Corollary~5.11]{Mal1}. As $v\ge \chi_{E\cup F}$, we have $C_X(E\cup F) \le \|v\|_\NX \le \|v\|_X + \|g\|_X < C_X(E) + \eps$. Letting $\eps\to 0$ shows that $C_X(E\cup F) \le C_X(E)$.
\end{proof}
%
% ---------------------------------------------------------
%
The following result generalizes~\cite[Proposition~1.4]{BjoBjoSha} in a similar fashion as Bj\"orn, Bj\"orn, and Lehrb\"ack~\cite[Proposition~4.7]{BjoBjoLeh}. It shows that in locally compact (and hence in proper) metric measure spaces, the Sobolev capacity $C_X$ (and hence also $\widetilde{C}_{X,r}$) is an outer capacity at least for zero sets, given that lsc majorants provide good estimates of the function norm.
\begin{pro}
\label{pro:Cap-out-regular_0}
Assume that $\Pcal$ is locally compact and that $X$ is a quasi-Banach function lattice possessing the \emph{Vitali--Carath\'eodory property}, i.e., it satisfies \eqref{eq:VitaliCaratheodory}. Let $E\subset \Pcal$ with $C_X(E) = 0$. Then for every $\eps>0$, there is an open set $U\supset E$ with $C_X(U) < \eps$.
\end{pro}
%
% ---------------------------------------------------------
%
\begin{proof}
Suppose first that there is an open set $G \supset E$ such that $\ittoverline{G}$ is compact.

Let $\eps>0$. Since $C_X(E) = 0$, we have $\|\chi_E\|_\NX = 0$ and there is an upper gradient $g\in X$ of $\chi_E$ such that $\|g\|_X<\eps$. By \eqref{eq:VitaliCaratheodory}, we can find $v,\rho \in X \cap \lsc(\Pcal)$ that satisfy $v\ge \chi_E$ and $\rho \ge g$ everywhere in $\Pcal$, while $\|v\|_X < \eps$ and $\|\rho\|_X<\eps$. Let $\tilde{v} = v \chi_G$. Since $G$ is open and contains $E$, we have that $\chi_E \le \tilde{v} \in X \cap \lsc(\Pcal)$ and $\|\tilde{v}\|_X < \eps$ as well.

Let $V = \{x\in\Pcal: \tilde{v}(x) > 1/2\}$. Then, $E \subset V \subset G$ and $V$ is open. Furthermore, $\|\chi_V\|_X \le 2 \|\tilde{v}\|_X < 2\eps$.
Let
\[ 
  u(x) = \min \biggl\{1, \inf_\gamma \int_\gamma (\rho + 1)\,ds\biggr\},
\]
where the infimum is taken over all (including constant) curves connecting $x$ to the closed set $\Pcal \setminus V$. Then, $u|_{\ittoverline{G}}\in \lsc(\ittoverline{G})$ by Bj\"orn, Bj\"orn and Shanmugalingam~\cite[Lem\-ma~3.3]{BjoBjoSha} since $\ittoverline{G}$, being compact, is a proper metric space and $\rho+1$ is bounded away from zero. Consequently, $u\in\lsc(\Pcal)$ as $u\equiv 0$ on $\Pcal \setminus V \supset \Pcal \setminus G$.

Now, let $U=\{ x\in \Pcal: u(x)>1/2\}$. The set $U$ is open due to the semicontinuity of~$u$. We can show that $u = 1$ on $E$, whence $E \subset U$. Indeed, let $\gamma$ be a curve connecting arbitrary $x \coloneq \gamma(0) \in E$ with $y \coloneq \gamma(l_\gamma) \in\Pcal \setminus V$. Then, $\int_\gamma (\rho + 1)\,ds \ge |\chi_E(x) - \chi_E(y)| + l_\gamma > 1$ as $\rho$ is an upper gradient of $\chi_E$. Furthermore, $u\le \chi_V$ and $(\rho+1)\chi_V$ is an upper gradient of $u$ due to~\cite[Lemmata~3.1 and~3.2]{BjoBjoSha}.
We can therefore estimate the capacity
\begin{align*}
  C_X(U) & \le 2 \|u\|_\NX \le 2 \bigl(\|\chi_V\|_X + \| (\rho+1)\chi_V\|_X\bigr) \le 2 \bigl(\|\chi_V\|_X + \cconc(\|\rho \chi_V\|_X + \|\chi_V\|_X)\bigr) \\
	&\le 2 \bigl((1+\cconc) \|\chi_V\|_X + \cconc \| \rho \|_X\bigr) \le 2 \bigl((1+\cconc) 2\eps + \cconc \eps \bigr) \le 10 \cconc \eps.
\end{align*}

If no open neighborhood of $E$ has a compact closure, then we can apply separability and the local compactness of $\Pcal$ to write $E=\bigcup_{n=1}^\infty E_n$ so that for each $n\in\Nbb$ there is an open set $G_n \supset E_n$ with a compact closure. In particular, $C_X(E_n) = 0$. By the previous part of the proof, we can find open sets $U_n \supset E_n$ with $C_X(U_n) < \eps/(2\cconc)^n$. Let now $U = \bigcup_{n=1}^\infty U_n$. Then, $U$ is open and $C_X(U) \le \sum_{n=1}^\infty \cconc^n C_X(U_n) < \eps$ by the $\sigma$-quasi-additivity of $C_X$.
\end{proof}
\begin{rem}
In the previous claim, it in fact suffices to assume that there exists an open set $G\supset E$ that is locally compact instead of requiring that the entire space $\Pcal$ is locally compact. It is, however, currently unknown whether such local compactness is really necessary. On the other hand, the Vitali--Carath\'eodory property is crucial. In view of Proposition~\ref{pro:VitaliCarath}, it suffices that $X$ contains bounded functions with bounded support and these have absolutely continuous norm in $X$. In~\cite{BjoBjoMal}, Bj\"orn, Bj\"orn and Mal\'y have constructed a metric measure space $\Pcal$ and a function space $X=X(\Pcal)$ such that Propositions~\ref{pro:VitaliCarath} and~\ref{pro:Cap-out-regular_0} fail.
\end{rem}
%
%
%  SECTION 4: QUASI-CONTINUITY
%
%
\section{Quasi-continuity and its consequences}
\label{sec:quasicontinuity}
In this section, we study when Newtonian functions possess a Luzin-type property, the so-called quasi-continuity, which then leads to the fact that the Sobolev capacity is an outer capacity. The Sobolev capacity defined via the Newtonian quasi-norm characterizes the equivalence classes well, but in general it need not be an outer capacity then. Outside of the Newtonian setting, it is customary to introduce the Sobolev capacity so that it is an outer capacity by definition, cf.\@ Heinonen, Kilpel\"ainen and Martio~\cite[Definition~2.35]{HeiKilMar} or Kinnunen and Martio~\cite[Section~3]{KinMar}.

In Section~\ref{sec:density}, the quasi-continuity will help us establishing the density of compactly supported Lipschitz functions in $\NnX(\Omega)$, where $\Omega \subset \Pcal$ is open.
%
% ---------------------------------------------------------
%
\begin{df}
A function $u: \Pcal \to \overline{\Rbb}$ is \emph{weakly quasi-continuous} if for every $\eps>0$ there is a set $E \subset \Pcal$ with $C_X(E) < \eps$ such that $u\vert_{\Pcal \setminus E}$ is continuous. If the set $E$ can be chosen open for every $\eps>0$, then $u$ is \emph{quasi-continuous}.
\end{df}
When an extended real-valued function is said to be continuous, we mean that the function does not attain the values $\pm \infty$ and is in fact real-valued.

As the capacities $C_X$ and $\widetilde{C}_{X,r}$ are equivalent, it is insignificant whether the notion of (weak) quasi-continuity is defined using $C_X$ or $\widetilde{C}_{X,r}$.
%
% ---------------------------------------------------------
%
\begin{lem}
\label{lem:quasicont-preserved-qe}
Assume that for every $\eps$ and every $E\subset \Pcal$ with $C_X(E) = 0$, there is an open set $U\supset E$ with $C_X(U) < \eps$. If $u$ is quasi-continuous, then every $v$ that coincides with $u$ q.e.\@ is also quasi-continuous.
\end{lem}
In view of Proposition~\ref{pro:Cap-out-regular_0}, the hypotheses of the lemma are satisfied if $\Pcal$ is locally compact and if $X$ has the Vitali--Carath\'eodory property~\eqref{eq:VitaliCaratheodory} (which, in particular, it does by Proposition~\ref{pro:VitaliCarath} if $X$ contains all bounded functions with bounded support and these have absolutely continuous quasi-norm in $X$).
\begin{proof}
Let $\eps>0$. Define $E=\{x\in\Pcal: u(x) \neq v(x)\}$, so $C_X(E) = 0$. Thus, we can find an open set $U \supset E$ with $C_X(U) < \eps/2\cconc$.
Since $u$ is quasi-continuous, there is an open set $V$ with $C_X(V) < \eps/2\cconc^2$ such that $u\vert_{\Pcal \setminus V}$ is continuous. Let $G = U \cup V$. Then, $C_X(G) < \eps$ and $v\vert_{\Pcal \setminus G} = u\vert_{\Pcal \setminus G}$ is continuous. Hence, $v$ is quasi-continuous.
\end{proof}
%
% ---------------------------------------------------------
%
Next, we will see that if the Sobolev capacity is an outer capacity, then the distinction between weak quasi-continuity and quasi-continuity is not needed.
\begin{pro}
\label{pro:qc=wqc}
Suppose that $C_X$ or $\widetilde{C}_{X,r}$ is an outer capacity. Then, a function is quasi-continuous if and only if it is weakly quasi-continuous.
\end{pro}
\begin{proof}
Quasi-continuous functions are trivially weakly quasi-continuous.

Let $u$ be weakly quasi-continuous and let $\eps>0$. Then, there is a set $E\subset \Pcal$ with $C_X(E) < \eps$ (resp.\@ $\widetilde{C}_{X,r}(E) < \eps$) such that $u|_{\Pcal \setminus E}$ is continuous. Since $C_X$ (resp.\@ $\widetilde{C}_{X,r}$) is an outer capacity, there is an open set $G \supset E$ with $C_X(G) < \eps$ (resp.\@ $\widetilde{C}_{X,r}(G) < \eps$). Then, $u|_{\Pcal \setminus G}$ is also continuous, whence $u$ is quasi-continuous.
\end{proof}
%
% ---------------------------------------------------------
%
For functions that are absolutely continuous along $\Mod_X$-a.e.\@ curve, the equality a.e.\@ implies equality on a larger set, namely, q.e.\@ (see~\cite[Proposition~6.12]{Mal1}). Similarly, we have the following result for quasi-continuous functions.
%
% ---------------------------------------------------------
%
\begin{pro}
\label{pro:qc_ae=qe}
Suppose that both $u$ and $v$ are quasi-continuous. If $u=v$ a.e., then $u=v$ q.e.
\end{pro}
%
% ---------------------------------------------------------
%
\begin{proof}[Sketch of proof]
Kilpel\"ainen~\cite{Kil} has proven the claim for abstract outer capacities that satisfy the rigidity condition of Proposition~\ref{pro:cap_rigid} under the assumption that $u$ and $v$ are weakly quasi-continuous. His proof works verbatim if the hypotheses of the capacity being outer and the weak quasi-continuity of $u$ and $v$ are replaced by the hypothesis that $u$ and $v$ are quasi-continuous.
\end{proof}
%
% ---------------------------------------------------------
%
Since a Egorov-type convergence theorem (see~\cite[Corollary~7.2]{Mal1}) holds in Newtonian spaces based on an arbitrary quasi-Banach function lattice $X$, we will obtain that Newtonian functions are weakly quasi-continuous provided that continuous functions are dense in $\NX$.

Sufficient conditions for density of (Lipschitz) continuous functions in $\NX$ have been discussed in~\cite{Mal3} using the connection between Haj\l asz gradients, fractional sharp maximal functions, and (weak) upper gradients in doubling $p$-Poincar\'e spaces (see Definition~\ref{df:PI} below). Roughly speaking, Lipschitz functions are dense in $\NX$ if a certain maximal operator of Hardy--Littlewood type satisfies weak norm estimates and the quasi-norm of $X$ is absolutely continuous.

In Theorem~\ref{thm:noncompl-cont} below, it will be shown that Newtonian functions have continuous representatives (with respect to equality q.e.\@) if $\Pcal$ supports a certain Poincar\'e inequality, $\mu$ is doubling and the quasi-norm of $X$ is sufficiently restrictive in comparison with the ``dimension of the measure''. In that case, the continuous functions are trivially dense in $\NX$.

Using the reflexivity of $N^{1,p} \coloneq  N^1 L^p$ that was established by Ambrosio, Colombo and Di~Marino~\cite{AmbColDiM}, we can deduce from Ambrosio, Gigli and Savar\'e~\cite{AmbGigSav} that Lipschitz functions are dense in $N^{1,p}$ for $p\in (1, \infty)$ if $\Pcal$ is compact, endowed with a doubling metric. In particular, neither a Poincar\'e inequality, nor a doubling property of the measure is needed.
\begin{pro}
\label{pro:newt.fcns-are-wqcont}
If continuous functions are dense in $\NX$, then every $u\in\NX$ has a quasi-continuous representative $\tilde{u}=u$ q.e. Hence, $u$ is weakly quasi-continuous.
\end{pro}
%
% ---------------------------------------------------------
%
\begin{proof}
Let $u\in\NX$ be approximated by a sequence $\{u_k\}_{k=1}^\infty\subset \Ccal(\Pcal)\cap \NX$ so that $u_k \to u$ in $\NX$ as $k\to\infty$. By~\cite[Corollary~7.2]{Mal1}, there is $\tilde{u}\in\NX$ such that $\tilde{u}=u$ q.e.\@ and for every $\eps>0$ there exists an open set $U_\eps$ with $C_X(U_\eps)<\eps$ such that a subsequence $\{u_{k_j}\}_{j=1}^\infty$ converges uniformly to $\tilde{u}$ on $\Pcal \setminus U_\eps$. Hence, $\tilde{u}|_{\Pcal\setminus U_\eps}$ is continuous and $\tilde{u}$ is quasi-continuous. Writing $E = \{x\in\Pcal: u(x) \neq \tilde{u}(x)\}$, we have $C_X(U_\eps \cup E) < \eps$ by Lemma~\ref{lem:C_X-addnull} and $u\vert_{\Pcal \setminus(U_\eps \cup E)}$ is continuous whence $u$ is weakly quasi-continuous.
\end{proof}
%
% ---------------------------------------------------------
%
In Proposition~\ref{pro:Cap-out-regular_0}, we saw that $C_X$ is an outer capacity for zero sets under certain hypotheses. As a consequence, we obtain that functions with a quasi-continuous representative are in fact quasi-continuous by Lemma~\ref{lem:quasicont-preserved-qe}. Then, Newtonian functions are quasi-continuous by Proposition~\ref{pro:newt.fcns-are-wqcont} provided that they can be approximated by continuous functions. Hence, we have the following result.
\begin{cor}
\label{cor:newt.fcns-are-qcont}
Assume that $\Pcal$ is locally compact and $X$ is a quasi-Banach function lattice with the Vitali--Carath\'eodory property \eqref{eq:VitaliCaratheodory}. In particular, we may assume that $X$ contains characteristic functions of all bounded sets and these have absolutely continuous norm in $X$. If continuous functions are dense in $\NX$, then every $u\in \NX$ is quasi-continuous.
\end{cor}
%
% ---------------------------------------------------------
%
The original idea of Proposition~\ref{pro:newt.fcns-are-wqcont} and Corollary~\ref{cor:newt.fcns-are-qcont} for $X=L^p$ under considerably stronger assumptions can be traced back to Shanmugalingam~\cite{Sha}, whose result was later generalized by Bj\"orn, Bj\"orn, and Shanmugalingam~\cite{BjoBjoSha}.
%
% ---------------------------------------------------------
%
\begin{pro}
\label{pro:wqc-continuous}
Suppose that there is a cover $\Pcal = \bigcup_{k=1}^\infty \Pcal_k$, where $\Pcal_k$ is open and that for every $k\in\Nbb$, there is $\delta_k>0$ such that $C_X(\{x\}) \ge \delta_k$ for each $x\in\Pcal_k$. Then, weakly quasi-continuous functions are continuous.

In particular, if all functions in $\NX$ are weakly quasi-continuous (which holds, e.g., if continuous functions are dense in $\NX$), then $\NX \subset \Ccal(\Pcal)$.
\end{pro}
\begin{proof}
Set $\eps_k = \delta_k/2$ for every $k\in\Nbb$. Let $u$ be weakly quasi-continuous. Then, there is $E_k$ with $C_X(E_k)<\eps_k$ such that $u|_{\Pcal\setminus E_k}$ is continuous. Consequently, $\Pcal_k \cap E_k = \emptyset$ because every $x\in \Pcal_k$ satisfies $C_X(\{x\}) \ge \delta_k > C_X(E_k)$. Thus, $u|_{\Pcal_k} \in \Ccal(\Pcal_k)$. Since $\Pcal_k$ is open, $u$ is continuous at every point of $\Pcal_k$.

Finally, $u$ is continuous everywhere in $\Pcal$ as $\Pcal = \bigcup_{k=1}^\infty \Pcal_k$. Hence, $u\in\Ccal(\Pcal)$.

If continuous functions are dense in $\NX$, then all functions in $\NX$ are weakly quasi-continuous by Proposition~\ref{pro:newt.fcns-are-wqcont}.
\end{proof}
%
% ---------------------------------------------------------
%
So far, we have seen that $C_X$ (or $\widetilde{C}_{X,r}$) being an outer capacity on zero sets implies that Newtonian functions are quasi-continuous (under some additional assumptions). The next proposition shows that the converse is actually stronger. Namely, if Newtonian functions are quasi-continuous, then $\widetilde{C}_{X,r}$ is an outer capacity on all sets (without any additional assumptions). An analogous result for $\widetilde{C}_{L^p, p}$ with $p\in[1, \infty)$ was given in~\cite{BjoBjoSha}.
\begin{pro}
\label{pro:C_X-outer}
Assume that all functions in $\NX$ are quasi-continuous. Then, $\widetilde{C}_{X,r}$ is an \emph{outer capacity}, i.e., $\widetilde{C}_{X,r}(E) = \inf \widetilde{C}_{X,r}(G)$ for every $E\subset \Pcal$, where the infimum is taken over all open sets $G\supset E$. Moreover, if $X$ is normed, then $C_X$ is an outer capacity.
\end{pro}
%
% ---------------------------------------------------------
%
\begin{proof}
If $X$ is normed, then $C_X = \widetilde{C}_{X,r}$ with $r=1$. Hence, it suffices to prove that $\widetilde{C}_{X,r}$ is an outer capacity.

If $\widetilde{C}_{X,r}(E) = \infty$, then the claim is trivial. Suppose therefore that $\widetilde{C}_{X,r}(E) < \infty$.
Let $\eps \in (0,1)$ and $u\ge \chi_E$ be such that $\|u\|_X^r + \|g_u\|_X^r < \widetilde{C}_{X,r}(E)^r + \eps$. Due to the quasi-continuity of $u$, there is an open set $V$ with $\widetilde{C}_{X,r}(V) < \eps$ such that $u\vert_{\Pcal\setminus V}$ is continuous. Thus, $U \coloneq \{x\in \Pcal \setminus V: u(x) > 1-\eps\}$ is open in $\Pcal \setminus V$ and $U\cup V$ is open in $\Pcal$. Let $v\ge \chi_V$ be such that $\|v\|_X^r + \|g_v\|_X^r < \eps$. Let
\[
  w = \frac{u}{1-\eps} + v.
\]
Then, $w \ge \chi_{U\cup V} \ge \chi_E$. Consequently,
\begin{align*}
  \widetilde{C}_{X,r}(E)^r & \le \inf_{\substack{G\supset E \\ G\textup{ open}}} \widetilde{C}_{X,r}(G)^r \le \widetilde{C}_{X,r}(U\cup V)^r \le \|w\|_X^r + \|g_w\|_X^r \\
  & \le \frac{\|u\|_X^r + \|g_u\|_X^r}{(1-\eps)^r} + \|v\|_X^r + \|g_v\|_X^r
  < \frac{\widetilde{C}_{X,r}(E)^r + \eps}{(1-\eps)^r} + \eps.
\end{align*}
The last expression tends to $\widetilde{C}_{X,r}(E)^r$ as $\eps\to 0$, which finishes the proof.
\end{proof}
%
% ---------------------------------------------------------
%
The following proposition quantifies the difference between a.e.\@ and q.e.\@ equivalence classes in Newtonian spaces. Namely, $\NX$ contains only the ``good'' representatives of the functions that lie in an a.e.\@ equivalence class of a Newtonian function.
\begin{pro}
Assume that $\Pcal$ is locally compact and that continuous functions are dense in $\NX$. Suppose further that $X$ has the Vitali--Carath\'eodory property \eqref{eq:VitaliCaratheodory}. In particular, it suffices to assume that $\chi_B \in X$ for every bounded set $B\subset \Pcal$ and it satisfies \ref{df:AC}. Let $u: \Pcal \to \overline{\Rbb}$ be such that $u=v$ a.e.\@ in $\Pcal$ for some function $v \in \NX$. Then, the following are equivalent:
\begin{enumerate}
	\item \label{it:ae-NX} $u \in \NX$;
	\item \label{it:ae-ACC} $u \circ \gamma$ is absolutely continuous for $\Mod_X$-a.e.\@ curve $\gamma$;
	\item \label{it:ae-wqc}  $u$ is weakly quasi-continuous;
	\item \label{it:ae-qc} $u$ is quasi-continuous.
\end{enumerate}
\end{pro}
\begin{proof}
Proposition~\ref{pro:VitaliCarath} gives \eqref{eq:VitaliCaratheodory} if $\chi_B\in X$ satisfies \ref{df:AC} whenever $B\subset\Pcal$ is bounded.

The equivalence \ref{it:ae-NX}\,$\Leftrightarrow$\,\ref{it:ae-ACC} was established in~\cite[Proposition~6.18]{Mal1}, without any assumptions on $\Pcal$.

The implication \ref{it:ae-NX}\,$\Rightarrow$\,\ref{it:ae-wqc} is shown in Proposition~\ref{pro:newt.fcns-are-wqcont}.

The equivalence \ref{it:ae-wqc}\,$\Leftrightarrow$\,\ref{it:ae-qc} follows by Proposition~\ref{pro:qc=wqc}, whose hypotheses are satisfied due to Corollary~\ref{cor:newt.fcns-are-qcont} and Proposition~\ref{pro:C_X-outer}.

In order to show that \ref{it:ae-qc}\,$\Rightarrow$\,\ref{it:ae-NX}, assume that $u$ is quasi-continuous. By \ref{it:ae-NX}\,$\Rightarrow$\,\ref{it:ae-qc}, we have that $v$ is quasi-continuous as $v\in \NX$. Then, $u=v$ q.e.\@ by Proposition~\ref{pro:qc_ae=qe}. Therefore, $\|u-v\|_\NX = 0$ by~\cite[Proposition~6.15]{Mal1}, which yields that $u\in \NX$.
\end{proof}
%
%
%  SECTION 5: DENSITY OF LOCALLY LIPSCHITZ FUNCTIONS
%
%
\section{Density of locally Lipschitz functions}
\label{sec:density}
The aim of this section is to prove that locally Lipschitz functions are dense in $\NX(\Omega)$, whenever $\Omega \subset \Pcal$ is open, provided that (locally) Lipschitz functions are dense in $\NX(\Pcal)$. Note that we will not pose any assumptions on $\Omega$ besides being open. We will however need $\Pcal$ to be proper and $X$ to have absolutely continuous norm that satisfies \ref{df:BFS.finmeasfinnorm}. Here, we generalize the results of~\cite{BjoBjoSha}, where $X$ was just $L^p$.

It has been shown in~\cite[Section~3]{Mal3} that Newtonian functions can be approximated by their truncations if the (quasi)norm of $X$ is absolutely continuous. We can extend this result if all Newtonian functions are quasi-continuous. Namely, Newtonian functions that vanish outside of a measurable set $\Theta$ can be approximated by bounded functions whose support is a bounded subset of $\Theta$. The case $X=L^p$ with an open $\Theta$ was discussed in~\cite{BjoBjoSha}, where the fundamental idea came from~\cite{ShaH}.
%
% ---------------------------------------------------------
%
\begin{lem}
\label{lem:quasicont_bdd-bddspt-dense}
Let $X$ be a quasi-Banach function lattice with absolutely continuous quasi-norm.
Assume that all functions in $\NX$ are quasi-continuous, Then, every function in $\NnX(\Theta)$ can be approximated in $\NX$ by bounded functions with bounded support lying in $\Theta$.
\end{lem}
%
% ---------------------------------------------------------
%
\begin{proof}
Let $u\in\NnX(\Theta)$ with an $X$-weak upper gradient $g\in X$. It has been shown in~\cite[Corollary~3.4]{Mal3} that the truncations of $u$ get arbitrarily close to $u$ in $\NX$. Therefore, we may assume that $u$ is bounded.

Next, we will show that we may assume that $u$ has bounded support. Let us fix $x_0\in \Pcal$ and write $\psi_n(x) = (1-\dist(x, B(x_0, n)))^+$ for $n\in\Nbb$. We want to prove that $u\psi_n \to u$ in $\NX$ as $n\to \infty$. Since $\tilde{g} = \chi_{A(x_0, n, n+1)}$ is an upper gradient of $1-\psi_n(x)=\min\{1, \dist(x, B(x_0, n))\}$, where $A(x_0, n, n+1)$ is the closed annulus $\{x\in\Pcal: n\le \dd(x, x_0) \le n+1\}$, the product rule (Theorem~\ref{thm:product_rule}) yields that the function $g_n = (1-\psi_n)g + \chi_{A(x_0, n, n+1)}u$ is an $X$-weak upper gradient of $(1-\psi_n)u$. Moreover, $g_n \le (u + g) \chi_{\Pcal \setminus B(x_0, n)}$.
Hence,
\begin{align*}
  \| u - u\psi_n \|_\NX & \le \|u-u\psi_n\|_X + \|g_n\|_X \\
  & \le \| u \chi_{\Pcal \setminus B(x_0, n)}\|_X + \|(u+g)\chi_{\Pcal \setminus B(x_0, n)}\|_X \to 0 \quad \mbox{as $n\to \infty$}
\end{align*}
due to the absolute continuity of the quasi-norm of $X$. Therefore, we do not lose any generality if we suppose that $u$ has bounded support.

Since $u$ is quasi-continuous, there are open sets $U_k\subset \Pcal$, $k\in\Nbb$, such that $u\vert_{\Pcal\setminus U_k}$ is continuous while $C_X(U_k) \to 0$ as $k\to\infty$. Thus, there exist functions $w_k\in \NX$ such that $\chi_{U_k} \le w_k \le 1$ and $\|w_k\|_\NX \to 0$ as $k\to\infty$. 
By~\cite[Corollary~7.2]{Mal1}, we may assume that $w_k \to 0$ q.e., passing to a subsequence if necessary. The sets $G_k \coloneq U_k \cup \{x\in \Pcal \setminus U_k: u(x)<1/k\}$ are open in $\Pcal$, whence $\Pcal \setminus G_k \subset \Theta$ is closed. Let
\[
  \eta_k(t) =
  \begin{cases}
    0& \mbox{for }|t|<1/k,\\
    2(|t| - 1/k)\sgn t & \mbox{for }1/k \le |t| \le 2/k,\\
    t& \mbox{for }|t|>2/k.
  \end{cases}
\]
By the chain rule (Theorem~\ref{thm:chain_rule}), we obtain that $g_{\eta_k\circ u - u} = g_u \chi_{\{0<|u|<2/k\}}$ a.e.\@ as the function $t \mapsto \eta_k(t) - t$ is $1$-Lipschitz and supported in $[-2/k, 2/k]$. The absolute continuity of the norm of $X$ now yields
\[
  \| \eta_k \circ u - u \|_\NX \le \| u \chi_{\{0<|u| < 2/k\}}\|_X + \|g_u \chi_{\{0<|u| < 2/k\}}\|_X \to 0\quad\mbox{as }k\to\infty
\]
since $\bigcap_{k=1}^\infty \{x\in\Pcal: 0<|u(x)| < 2/k\}= \emptyset$. As $\eta_k$ is $2$-Lipschitz, we can estimate $g_{\eta_k \circ u} \le 2 g_u$ a.e. Let $u_k = (1-w_k) (\eta_k\circ u)$. Then, $u_k$ is supported within $\Pcal \setminus G_k$ and the product rule gives that
\begin{align*}
  \| u_k - \eta_k \circ u\|_\NX & = \|(\eta_k \circ u) w_k\|_\NX \\
  & \le \|(\eta_k \circ u) w_k\|_X + \| (\eta_k \circ u) g_{w_k}+ g_{\eta_k \circ u} w_k\|_X \\
  & \le \|\eta_k \circ u\|_{L^\infty} (\|w_k\|_X + \cconc\|g_{w_k}\|_X) + 2\cconc\|g_{u} w_k\|_X. 
\end{align*}
For an arbitrary $\eps > 0$, we obtain that  $\|g_{u} w_k\|_X \le \cconc (\eps \|g_u\|_X + \|g_u \chi_{E_k(\eps)}\|_X)$, where $E_k(\eps) = \{x\in \Pcal: w_k(x) > \eps\}$. Then, $\limsup_{k\to \infty} \|g_{u} w_k\|_X \le \cconc \eps \|g_u\|_X$ by the absolute continuity of the quasi-norm of $X$ since $\meas{\bigcap_{k=1}^\infty E_k(\eps)} = 0$. Letting $\eps \to 0$, we see that $\|g_u w_k\|_X \to 0$ as $k\to \infty$. 
The choice of $w_k$ ensures that $\|w_k\|_X + \|g_{w_k}\|_X = \|w_k\|_\NX \to 0$ as $k\to \infty$. Thus, $\| u_k - \eta_k \circ u\|_\NX \to 0$ as $k\to \infty$.
\end{proof}
%
% ---------------------------------------------------------
%
We can go even further if $\Pcal$ is a proper metric measure space. If $\Omega \subset \Pcal$ is open and if locally Lipschitz functions are dense in $\NX$, then the Newtonian functions that vanish outside of $\Omega$ can be approximated (in the norm of $\NX$) by Lipschitz functions that are compactly supported within $\Omega$.
%
% ---------------------------------------------------------
%
\begin{pro}
\label{pro:Lipc-dense-NX0}
Suppose that $\Pcal$ is proper and $\Omega \subset \Pcal$ is open. Assume that $X$ is a quasi-Banach function lattice that has absolutely continuous quasi-norm and satisfies~\ref{df:BFS.finmeasfinnorm}.
If locally Lipschitz functions are dense in $\NX$, then $\overline{\Lipc(\Omega)}=\NnX(\Omega)$.
\end{pro}
%
% ---------------------------------------------------------
%
\begin{proof}
Since bounded functions with bounded support are contained in $X$ by \ref{df:BFS.finmeasfinnorm}, we immediately obtain that $\overline{\Lipc(\Omega)} \subset \itoverline{\NnX(\Omega)}$. Now, we will show that $\NnX(\Omega)$ is a closed subset of $\NX$. Let $\{u_k\}_{k=1}^\infty$ be a Cauchy sequence in $\NnX(\Omega)$. Then, there exists $u \in \NX$ such that $\| u_k - u\|_\NX$ as $k\to \infty$ since $\NX$ is complete by~\cite[Theorem~7.1]{Mal1}. By passing to a subsequence if needed, we have that $u_k \to u$ pointwise q.e.\@ in $\Pcal$ by~\cite[Corollary~7.2]{Mal1}. Hence, $u=0$ q.e.\@ in $\Pcal \setminus \Omega$. Let $\tilde{u} = u \chi_\Omega$. Then, $\tilde{u} = u$ q.e.\@ in $\Pcal$. Therefore, $\|\tilde{u} - u\|_\NX = 0$ by~\cite[Proposition 6.15]{Mal1} and hence $\| u_k - \tilde{u}\|_\NX \to 0$ as $k\to \infty$, where $\tilde{u} \in \NnX$. Consequently, $\NnX(\Omega) = \itoverline{\NnX(\Omega)}$.

Let now $u\in \NnX(\Omega)$, $u\nequiv 0$. Due to Corollary~\ref{cor:newt.fcns-are-qcont} and Lemma~\ref{lem:quasicont_bdd-bddspt-dense}, we may assume that $u$ is bounded and has bounded support within $\Omega$. In fact, $\spt u$ is compact since $\Pcal$ is proper. Let $\eps > 0$ be arbitrary. There exists a locally Lipschitz function $v \in \NX$ such that $\|u-v\|_\NX < \eps$.
Next, we define $\eta(x) = (1-2 \dist(x, \spt u)/\delta)^+$, where $\delta = \min\{1, \dist(\spt u, \Pcal\setminus\Omega)\}$. The support of $\eta$ is compact in $\Omega$, and $g_\eta \le 2/\delta$ since $\eta$ is $2/\delta$-Lipschitz. Moreover, $\chi_{\spt u} \le \eta \le 1$. Consequently, $v\eta \in \Lipc(\Omega)$ and the product rule (Theorem~\ref{thm:product_rule}) gives $g_{v(1-\eta)} \le |v| g_\eta + g_v$. Since $1-\eta = 0$ on $\spt u$, Corollary~\ref{cor:u=v_gu=gv} yields that $g_{v(1-\eta)} = 0$ a.e.\@ on $\spt u$. Corollary~\ref{cor:u=v_gu=gv} further implies that $g_{u-v} = g_{-v}=g_v$ a.e.\@ outside of $\spt u$ because $u-v=-v$ there. Thus,
\begin{align*}
  \|v-v\eta\|_\NX & \le \|v \chi_{\Pcal\setminus \spt u}\|_X + \|(|v| g_\eta + g_v) \chi_{\Pcal\setminus \spt u}\|_X \\
  & \le \|v \chi_{\Pcal\setminus \spt u}\|_X + \cconc (\|v \chi_{\Pcal\setminus \spt u}\|_X \|g_\eta\|_{L^\infty} + \|g_v \chi_{\Pcal\setminus \spt u}\|_X) \\
  & \le \biggl(1+ \frac{2\cconc}{\delta}\biggr) \|v \chi_{\Pcal\setminus \spt u}\|_X + \cconc\|g_v \chi_{\Pcal\setminus \spt u}\|_X \\
  & \le \biggl(1+ \frac{2\cconc}{\delta}\biggr) \bigl(\| (u-v) \chi_{\Pcal\setminus \spt u}\|_X + \|g_{u-v} \chi_{\Pcal\setminus \spt u}\|_X\bigr)\,.
\end{align*}
Therefore,
\[
  \|v-v\eta\|_\NX \le \biggl(1+ \frac{2\cconc}{\delta}\biggr)\|u-v\|_\NX < \biggl(1+ \frac{2\cconc}{\delta}\biggr) \eps.
\]
The triangle inequality in $\NX$ now yields that
\[
  \|u-v\eta\|_\NX \le \cconc (\|u-v\|_\NX + \|v-v\eta\|_\NX) < 2\cconc \biggl(1+ \frac{\cconc}{\delta}\biggr) \eps,
\]
completing the proof of the inclusion $\NnX(\Omega) \subset \overline{\Lipc(\Omega)}$.
\end{proof}
%
% ---------------------------------------------------------
%
Finally, if we consider the space of Newtonian functions on an open subset $\Omega$ of a proper metric space $\Pcal$, then the density of locally Lipschitz functions in $\NX(\Pcal)$, implies the density in $\NX(\Omega)$. What makes this claim interesting is that we do not impose any other conditions on $\Omega$. In particular, it has been shown earlier that Lipschitz functions are dense in $\NX(\Pcal)$ if $\Pcal$ supports a $p$-Poincar\'{e} inequality and the maximal operator $M_p$ has certain bounds, but here we do not assume that $\Omega$ (as a metric subspace of $\Pcal$) is a $p$-Poincar\'{e} space nor that $\mu|_\Omega$ is doubling. On the other hand, we merely obtain density of locally Lipschitz functions, which is however not unexpected in view of~\cite[Examples~5.8--5.11]{BjoBjo}.
\begin{exa}[{\cite[Example 5.4]{BjoBjoSha}}]
Let $\Pcal$ be the slit disc $B(0,1) \setminus (-1, 0] \subset \Cbb = \Rbb^2$. Then, $f(z) = \max\{0, 2|z|-1\} \arg z$ belongs to $N^{1,p}(\Pcal) \setminus \overline{\Lip(\Pcal)}$ for all $p\ge 1$.
\end{exa}
%
% ---------------------------------------------------------
%
\begin{thm}
\label{thm:locLip-dense}
Under the hypotheses of Proposition~\ref{pro:Lipc-dense-NX0}, locally Lipschitz functions are dense in $\NX(\Omega)$.
\end{thm}
\begin{proof}
Let $u\in\NX(\Omega)$ and let $\eps>0$ be arbitrary. Since $\Pcal$ is proper, we may find an increasing sequence of open sets $\emptyset = \Omega_0 \neq \Omega_1 \Subset \Omega_2 \Subset \cdots \Subset \Omega$ so that $\Omega = \bigcup_{j=1}^\infty \Omega_j$. For each $j=1,2,\ldots$, choose $\eta_j \in \Lipc(\Omega_{j+1})$ such that $\chi_{\Omega_j} \le \eta_j \le 1$. Then, define $u_j = (u-\sum_{k=1}^{j-1} u_k) \eta_j$, which gives that $u_j \in \NnX(\Omega_{j+1} \setminus \overline{\Omega_{j-1}})$. We also obtain that $u = \sum_{j=1}^\infty u_j$ everywhere in $\Omega$.

By Proposition~\ref{pro:Lipc-dense-NX0}, there exists $v_j \in \Lipc(\Omega_{j+1} \setminus \overline{\Omega_{j-1}})$ such that $\|u_j - v_j\|_\NX \le (2\cconc)^{-j} \eps$ for every $j=1,2,\ldots$. Let $v=\sum_{j=1}^\infty v_j$. For every $x\in\Omega$, there is a neighborhood $U\ni x$ such that at most three terms in this sum are non-zero in $U$, whence $v$ is locally Lipschitz in $\Omega$. The triangle inequality now yields that
\[
  \|u-v\|_\NX(\Omega) \le \sum_{j=1}^\infty \cconc^j \|u_j - v_j\|_\NX \le \eps.
  \qedhere
\]
\end{proof}
%
%
%  SECTION 6: BOUNDEDNESS OF NEWTONIAN FUNCTIONS
%
%
\section{Boundedness of Newtonian functions}
\label{sec:boundedness}
In the setting of Sobolev spaces $W^{1,p}(\Rbb^n)$, it is well known that Sobolev functions are essentialy bounded and have continuous representatives if $p>n$. A finer distinction of function spaces is however needed for $p=n$, e.g., certain Zygmund or Lorentz norms can be used to ensure the boundedness. A similar result can be obtained for Newtonian functions if we introduce the notion of dimension of a doubling measure $\mu$. Therefore, we will assume that $\mu$ satisfies the \emph{doubling condition} in this (as well as in the next) section.

In $\Rbb^n$ with the Lebesgue measure $\lambda^n$, we have $\lambda^n(2B) = 2^n \lambda^n(B)$ for every ball $B\subset \Rbb^n$. The doubling condition gives $\mu(2B) \le c_\dbl \mu(B) = 2^{\log_2 c_\dbl} \mu(B)$. Even though $\log_2 c_\dbl$ can play the role of the dimension, it need not be sharp for the results about boundedness and continuity of Newtonian functions. It can be easily shown (see, e.g.,~\cite[Lemma 3.3]{BjoBjo}) that for every $s \ge \log_2 c_\dbl$ there is $c_s > 0$ such that
\begin{equation}
  \label{eq:dimension}
  \frac{\meas{B(y,r)}}{\meas{B(x,R)}} \ge c_s \biggl( \frac rR \biggr)^s
\end{equation}
whenever $0<r\le R$, $x\in \Pcal$, and $y \in B(x,R)$. Considering a simple example of weighted $\Rbb^n$ with a non-constant weight $w\in L^\infty(\Rbb^n)$ such that $1/w \in L^\infty(\Rbb^n)$, we see that \eqref{eq:dimension} holds with $s=n < \log_2 c_\dbl$ and $c_s = 1/\|w\|_\infty \|1/w\|_\infty$. Therefore, the dimension will be replaced by $s \le \log_2 c_\dbl$, preferably as small as possible, such that \eqref{eq:dimension} is satisfied. Note however that the set of admissible exponents $s$ may be open, see e.g.\@ Bj\"orn, Bj\"orn and Lehrb\"ack~\cite[Example~3.1]{BjoBjoLeh}. It is insignificant for the notion of dimension whether we require that \eqref{eq:dimension} holds for all $y \in B(x,R)$ or only for $y=x$ (i.e., only for concentric balls) since
\[
  c_\dbl^{-1} \frac{\meas{B(y,r)}}{\meas{B(y,R)}} \le \frac{\meas{B(y,r)}}{\meas{B(x,R)}} = \frac{\meas{B(y,r)}}{\meas{B(y,R)}} \frac{\meas{B(y,R)}}{\meas{B(x,R)}} \le c_\dbl \frac{\meas{B(y,r)}}{\meas{B(y,R)}}.
\]

If $\Pcal$ is connected, then there are $c_\sigma > 0$ and $0<\sigma \le s$ such that
\begin{equation}
   \label{eq:dimension2}
  \frac{\mu(B(y,r))}{\mu(B(x,R))} \le c_\sigma \biggl( \frac{r}{R}\biggr)^{\sigma}
\end{equation}
whenever $0<r\le R < 2\diam \Pcal$, $x\in \Pcal$, and $y \in B(x,R)$, see~\cite[Corollary~3.8]{BjoBjo}. Similarly as above, if \eqref{eq:dimension2} holds with some $\sigma$, then it holds with all $\sigma'\le \sigma$. The set of admissible exponents in \eqref{eq:dimension2} may be open, see~\cite[Example~3.1]{BjoBjoLeh}. Moreover, it may happen that $\sigma < s$ even if both $\sigma$ and $s$ are the best possible exponents (provided that these exist). The metric measure space is called \emph{Ahlfors $Q$-regular} if both \eqref{eq:dimension} and \eqref{eq:dimension2} are satisfied with $\sigma = s \eqcolon Q$. However, the Ahlfors regularity is a very restrictive condition that fails even in weighted $\Rbb^n$, unless the weight is bounded away both from zero and from infinity, see e.g.\@~\cite[Example 3.5]{BjoBjo}.

We will show that all Newtonian functions are locally essentially bounded (and have continuous representatives, which will be shown in the next section) provided that the function lattice $X$ is continuously embedded into $L^p_\loc$ for some $p>s$ or into $L^s(\log L)^{1+\eps}_\loc$ in the borderline case $p=s\ge 1$, where $s$ is the ``dimension of the measure'' given by \eqref{eq:dimension}, under the assumption that $\Pcal$ supports a $p$-Poincar\'e inequality. If a slightly stronger Poincar\'e inequality is assumed, then we will show that the embedding $X\emb L^{s,1}_\loc$ suffices to obtain local essential boundedness of functions in $\NX$ (and hence so does $X \emb L^s(\log L)^{1-1/s}$).
\begin{df}
\label{df:PI}
We say that $\Pcal$ supports a \emph{$p$-Poincar\'e inequality} with $p\in [1, \infty)$ if there exist constants $c_{\PI} > 0$ and $\lambda \ge 1$ such that
\begin{equation}
  \label{eq:PI}
  \fint_B |u-u_B|\,d\mu \le c_{\PI} \diam (B) \Biggl( \fint_{\lambda B} g^p\,d\mu\Biggr)^{1/p}
\end{equation}
for all balls $B\subset \Pcal$, for all $u\in L^1_\loc(\Pcal)$ and all upper gradients $g$ of $u$.
\end{df}
This inequality is sometimes called a \emph{weak $p$-Poincar\'e inequality} since we allow for $\lambda > 1$. Moreover, $\Pcal$ supports a $p$-Poincar\'e inequality if and only if \eqref{eq:PI} holds for all measurable functions $u$ and all $p$-weak upper gradients $g$ of $u$, where the left-hand side is interpreted as $\infty$ whenever $u_B$ is not defined or $u_B = \pm \infty$. We can also equivalently require that \eqref{eq:PI} holds for all $u\in L^\infty(\Pcal)$ and all ($p$-weak) upper gradients $g$ of $u$. These characterizations were shown in~\cite[Proposition 4.13]{BjoBjo}. If $X \emb L^p_\loc$, i.e., if $\|f \chi_B\|_{L^p} \le \cemb(B) \|f\chi_B\|_X$ for all balls $B\subset \Pcal$ and $f\in X$, then we may also require validity of the inequality for all $X$-weak upper gradients $g$ of $u$, which follows by~\cite[Lemma~5.6]{Mal1}.

If $\Pcal$ supports a $p$-Poincar\'e inequality for some $p\in [1, \infty)$, then it also supports a $q$-Poincar\'e inequality whenever $q \in [p, \infty)$ due to the H\"older inequality. It also follows that $\Pcal$ is connected (see, e.g., Shanmugalingam~\cite[p.~25]{ShaPhD}), whence $\mu$, being doubling, is non-atomic.

Both Zygmund and Lorentz spaces, which have been mentioned earlier, belong to a wide class of function spaces, the so-called \emph{\ri spaces}, i.e., Banach function spaces that are \emph{rearrangement-invariant}. Thus, they satisfy not only \ref{df:qBFL.initial}--\ref{df:BFL.locL1} with the modulus of concavity $\cconc = 1$, but also
\begin{enumerate}
  \renewcommand{\theenumi}{(RI)}
  \item \label{df:RI}
  if $u$ and $v$ are \emph{equimeasurable}, i.e.,
  \[
    \meas{\{x\in \Pcal: u(x) > t\}} = \meas{\{x\in \Pcal: v(x) > t\}} \quad\mbox{for all } t\ge0,
  \]
  then $\|u\|_X = \|v\|_X$.
\end{enumerate}
For a detailed treatise on \ri spaces, see Bennett and Sharpley~\cite{BenSha}.

For $f \in \Mcal(\Pcal, \mu)$, we define its \emph{distribution function $\mu_f$} and the \emph{decreasing rearrangement $f^*$} by
\begin{alignat*}{2}
  \mu_f(t) & = \meas{\{x\in \Pcal: |f(x)| > t\}}, \quad & t&\in[0, \infty), \\
  f^*(t) &= \inf\{ \tau\ge 0: \mu_f(\tau) \le t\}, &t&\in[0, \infty).
\end{alignat*}
The Cavalieri principle implies that $\|f\|_{L^1(\Pcal, \mu)} = \|\mu_f \|_{L^1(\Rbb^+, \lambda^1)} = \|f^*\|_{L^1(\Rbb^+, \lambda^1)}$.

We define the \emph{fundamental function} of a rearrangement-invariant quasi-Banach function lattice $X$ as $\phi_X(t) = \|\chi_{E_t}\|_X$, where $E_t \subset \Pcal$ is an arbitrary measurable set with $\meas{E_t} = \min\{t, \meas{\Pcal}\}$, $t>0$. Note that different spaces may very well have the same fundamental function, which is the case, e.g., of the Lebesgue $L^p$ and the Lorentz $L^{p,q}$ spaces as $\phi_{L^p}(t) = \phi_{L^{p,q}}(t) = t^{1/p}$ for $t < \meas{\Pcal}$ whenever $p\in[1, \infty)$ and $q\in[1, \infty]$. For another example, the Orlicz space $L^\Psi$ based on an $N$-function $\Psi$ has the fundamental function $\phi_{L^\Psi}(t) = 1/\Psi^{-1}(1/t)$.

We also have the continuous embedding of $X$ into the weak-$X$ space for every \ri space $X$, which can be expressed by the inequality $\sup_{t>0} u^*(t) \phi_X(t) \le \|u\|_X$.

In the next proposition, we will see that a $p$-Poincar\'e inequality gives not only an integral but also a supremal estimate for the oscillation of a Newtonian function, provided that $p$ is sufficiently large.
\begin{pro}
\label{pro:esssup_est_p>s}
Suppose that $\Pcal$ supports a $p$-Poincar\'e inequality with $p > s$ such that \eqref{eq:dimension} is satisfied. Suppose further that $X \emb L^p_\loc$. Let $B_0 \subset \Pcal$ be a fixed ball of radius $R>0$. Then, there is a constant $c_{B_0}>0$ such that for every ball $B \subset B_0$ of radius $r\in(0, R)$, we have
\begin{equation}
  \label{eq:CX-esssup_est_p>s}
  \mbox{$C_X$-}\esssup_{x\in B} |u(x) - u_B| \lesssim \frac{\cemb(2\lambda B) r}{\meas{2\lambda B}^{1/p}} \|g \chi_{2\lambda B}\|_X \le c_{B_0} r^{1-s/p} \|g \chi_{2\lambda B}\|_X
\end{equation}
whenever $g \in X$ is an $X$-weak upper gradient of $u\in \NX$. Moreover, we can estimate $c_{B_0} \approx \cemb(2\lambda B_0) (R^s/\meas{2 \lambda B_0})^{1/p}$.
\end{pro}
Note that we need to assume that $r \lesssim R$ for the second inequality in \eqref{eq:CX-esssup_est_p>s} as it may happen that $r \gg R$ even if $B \subset B_0$.
\begin{proof}
Let $B_0 = B(y, R) \subset \Pcal$ and $B = B(z, r)\subset B_0$ be arbitrary balls with $r<R$. Then,  $\mbox{$C_X$-}\esssup_{x\in B} |u(x) - u_B| = \esssup_{x\in B} |u(x) - u_B|$ by~\cite[Corollary~6.13]{Mal1} since $|u - u_B| \in DX$.

By~\cite[Proposition~4.27]{BjoBjo} and by the embedding $X \emb L^p_\loc$, we obtain that
\[
  \esssup_{x\in B} |u(x) - u_B| 
  \lesssim r \frac{\|g \chi_{2\lambda B}\|_{L^p}}{\meas{2\lambda B}^{1/p}} \le \frac{\cemb(2\lambda B) r}{\meas{2\lambda B}^{1/p}} \|g \chi_{2\lambda B}\|_X. 
\]
By applying \eqref{eq:dimension}, we see that $c_s (2\lambda r)^s / \meas{2\lambda B} \le (2\lambda R)^s / \meas{2\lambda B_0}$. Hence,
\[
 \frac{r}{\meas{2\lambda B}^{1/p}} = r^{1-s/p} \biggl(\frac{r^s}{\meas{2\lambda B}}\biggr)^{1/p} \le r^{1-s/p} \biggl(\frac{R^s}{c_s \meas{2\lambda B_0}}\biggr)^{1/p}\,.
\]
Moreover, $\cemb(2\lambda B) \le \cemb(2\lambda B_0)$, which yields the desired estimate
\[
  \frac{\cemb(2\lambda B) r}{\meas{2\lambda B}^{1/p}} \|g \chi_{2\lambda B}\|_X \le c_{B_0} r^{1-s/p} \|g \chi_{2\lambda B}\|_X\,.
  \qedhere
\]
\end{proof}
\begin{cor}
\label{cor:esssup_est_p>s}
Assume that $\Pcal$ supports a $p$-Poincar\'e inequality with $p > s$ such that~\eqref{eq:dimension} is satisfied. Let $\Omega \subset \Pcal$ be a fixed open set. Assume further that $X$ is an \ri space with fundamental function $\phi$ and that
\begin{equation}
  \label{eq:MX-Lp-emb_norm}
  c_\phi(\Omega) \coloneq \sup_{0<t<\meas{\Omega}} \phi(t) \biggl(\fint_0^t \frac{d\tau}{\phi(\tau)^p}\biggr)^{1/p} < \infty.
\end{equation}
Then, for every ball $B$ of radius $r>0$ such that $2\lambda B \subset \Omega$, we have
\[
  \mbox{$C_X$-}\esssup_{x\in B} |u(x) - u_B| \lesssim c_\phi(\Omega) r \frac{\|g \chi_{2\lambda B}\|_X}{\phi(\meas{2\lambda B})}\,,
\]
whenever $g \in X$ is an $X$-weak upper gradient of $u\in \NX$.
\end{cor}
\begin{proof}
In the proof of Proposition~\ref{pro:esssup_est_p>s}, we have seen that
\[
  \mbox{$C_X$-}\esssup_{x\in B} |u(x) - u_B| \lesssim r \biggl(\fint_{2\lambda B} g^p\,d\mu\biggr)^{1/p}.
\]
By the Cavalieri principle and by the embedding $X\emb\textrm{weak-}X$, it follows that
\begin{align*}
  \biggl(\fint_{2\lambda B} g^p\,d\mu\biggr)^{1/p} & = \biggl( \fint_0^{\meas{2\lambda B}} \biggl(\frac{(g \chi_{2\lambda B})^*(t) \phi(t)}{\phi(t)}\biggr)^p dt \biggr)^{1/p}\\
	& \le \phi(\meas{2\lambda B}) \biggl(\fint_0^{\meas{2\lambda B}} \frac{dt}{\phi(t)^p}\biggr)^{1/p}  \sup_{0<\tau<\meas{2\lambda B}} \frac{(g\chi_{2\lambda B})^*(\tau) \phi(\tau)}{\phi(\meas{2\lambda B})} \\
	& \le c_\phi(\Omega) \frac{\|g \chi_{2\lambda B}\|_X}{\phi(\meas{2\lambda B})}\,.
  \qedhere
\end{align*}
\end{proof}
%
% ---------------------------------------------------------
%
The previous proposition and corollary can be refined in the critical case when $p = s$; we will however need to assume that $X$ is embedded into the Zygmund space $L^s(\log L)^{\alpha}$ for some $\alpha>1$. That result will be further improved under somewhat stronger assumptions on $\Pcal$. 
\begin{df}
For $p\ge 1$ and $\alpha\ge0$, the \emph{Zygmund space} $L^p(\log L)^{\alpha}(E)$, where $E \subset \Pcal$ is measurable and $\meas{E}<\infty$, consists of the measurable functions $u: E\to\overline\Rbb$ such that
\[
  \| u \|^p_{L^p(\log L)^{\alpha}(E)} = \int_0^{\meas{E}} \biggl(u^*(t) \biggl(1+\log \frac{\meas{E}}{t}\biggr)^\alpha\biggr)^{p}\,dt < \infty.
\]
As an alternative, which is well-defined even if $\meas{E}=\infty$, we may use
\begin{equation}
  \label{eq:LpLog-eqnorm}
  \vvvert u \vvvert^p_{L^p(\log L)^{\alpha}(E)} = \int_0^{\meas{E}} \biggl(u^*(t) \biggl(1+\log^+ \frac 1t\biggr)^\alpha\biggr)^{p}\,dt,
\end{equation}
where $\log^+$ denotes the positive part of $\log$.
\end{df}
It is easy to see that $\vvvert u \vvvert_{L^p(\log L)^{\alpha}(E)} \approx \| u \|_{L^p(\log L)^{\alpha}(E)}$ where the constants depend on $\meas{E}<\infty$. Obviously, $L^p(\log L)^0 = L^p$. Furthermore, the Zygmund spaces are classical Lorentz spaces and they coincide with certain Orlicz classes whenever $\alpha \ge 1/p$. It is customary to drop the respective exponent in the notation $L^p(\log L)^\alpha (E)$ if $p=1$ or $\alpha = 1$.
%
% ---------------------------------------------------------
%
\begin{pro}
\label{pro:esssup_est_p=s}
Suppose that $\Pcal$ supports an $s$-Poincar\'e inequality with $s$ given by \eqref{eq:dimension}. Assume that $X \emb L^s(\log L)^\alpha_\loc\fcrim$ for some $\alpha \ge 1$ if $s=1$ and for some $\alpha > 1$ if $s>1$. Let $B_0 \subset \Pcal$ be a fixed ball of radius $R>0$. Then, there is a constant $c_{B_0}>0$ such that for every ball $B \subset B_0$ of radius $r\in(0, R)$, we have
\begin{equation}
  \label{eq:CX-esssup_est_p=d_Zyg}
  \mbox{$C_X$-}\esssup_{x\in B} |u(x) - u_B| \lesssim \frac{\cemb(2\lambda B) r}{\meas{2\lambda B}^{1/s}} \|g \chi_{2\lambda B}\|_X \le c_{B_0} \|g \chi_{2\lambda B}\|_X
\end{equation}
whenever $g \in X$ is an $X$-weak upper gradient of $u\in \NX$. Moreover, we can estimate $c_{B_0} \approx \cemb(2\lambda B_0) R/\meas{2 \lambda B_0}^{1/s}$.
\end{pro}
\begin{proof}
Let $B_0 = B(y, R) \subset \Pcal$ and $B = B(z, r)\subset B_0$ be arbitrary balls. It follows from~\cite[Corollary~6.13]{Mal1} that $\mbox{$C_X$-}\esssup_{x\in B} |u(x) - u_B| = \esssup_{x\in B} |u(x) - u_B|$ as $|u - u_B| \in DX$.

Let $x \in B$ be a Lebesgue point of $u$ and set $\widetilde{B} = B(x, r)$ and $\widetilde{B}_n = B(x, 2^{-n}r)$ for $n=0,1,2,\ldots$. 
Then,
\[
  u(x) = \lim_{n\to\infty} u_{\widetilde{B}_n} = u_{\widetilde{B}} + \sum_{n=0}^\infty \bigl( u_{\widetilde{B}_{n+1}} - u_{\widetilde{B}_n} \bigr)\,.
\]
Applying the triangle inequality, the doubling condition, and the $s$-Poincar\'{e} inequality (where an $X$-weak upper gradient $g\in X$ of $u$ may be used in light of the embedding $X \emb L^s(\log L)^\alpha_\loc \emb L^s_\loc$) yields
\begin{align*}
  |u(x) - u_{\widetilde{B}}| & \le \sum_{n=0}^\infty \bigl| u_{\widetilde{B}_{n+1}} - u_{\widetilde{B}_n} \bigr| \le \sum_{n=0}^\infty \fint_{\widetilde{B}_{n+1}} \bigl| u - u_{\widetilde{B}_n}\bigr|\,d\mu \\
  & \lesssim \sum_{n=0}^\infty \fint_{\widetilde{B}_n} \bigl| u - u_{\widetilde{B}_n}\bigr|\,d\mu
  \lesssim r \sum_{n=0}^\infty 2^{-n} \biggl(\fint_{\lambda\widetilde{B}_n} g^s\,d\mu\biggr)^{1/s}\,.
\end{align*}
We have $c_s \meas{\lambda \swidetilde{B}}\le 2^{ns} \meas{\lambda \swidetilde{B}_n}$ by \eqref{eq:dimension}, whence
\[
  \sum_{n=0}^\infty 2^{-n} \biggl(\fint_{\lambda\widetilde{B}_n} g^s\,d\mu\biggr)^{1/s} \lesssim \frac{1}{\meass{\lambda \widetilde{B}}^{1/s}} \sum_{n=0}^\infty \biggl(\int_{\lambda\widetilde{B}_n} g^s\,d\mu\biggr)^{1/s}.
\]
Let us, for the sake of brevity, write $\tilde{g} = g \chi_{\lambda \widetilde{B}}$. Since $\Pcal$ is connected due to the Poincar\'e inequality, we have that $\meas{\lambda \swidetilde{B}_n} \le c_\sigma 2^{-n\sigma} \meas{\lambda \swidetilde B}$ for some $0<\sigma\le s$ and $c_\sigma\ge 1$ by \eqref{eq:dimension2}. Let $A = c_\sigma \meas{\lambda \swidetilde B}$. Applying the Hardy--Littlewood inequality and replacing $\tilde{g}$ by its decreasing rearrangement, we obtain that
\[
\sum_{n=0}^\infty \biggl(\int_{\lambda\widetilde{B}_n} g^s\,d\mu\biggr)^{1/s}
 \le \sum_{n=0}^\infty \biggl(\int_0^{\meas{\lambda\swidetilde{B}_n}} \tilde{g}^*(t)^s\,dt\biggr)^{1/s} \le \sum_{n=0}^\infty \biggl(\int_0^{2^{-n\sigma} A} \tilde{g}^*(t)^s\,dt\biggr)^{1/s}.
\]
If $s=1$, then splitting the integration domain dyadically gives that
\[
  \sum_{n=0}^\infty \int_0^{2^{-n\sigma} A} \tilde{g}^*(t)\,dt = \sum_{n=0}^\infty \sum_{j=n}^\infty  \int_{2^{-(j+1)\sigma} A}^{2^{-j\sigma} A} \tilde{g}^*(t)\,dt = \sum_{j=0}^\infty (j+1) \int_{2^{-(j+1)\sigma} A}^{2^{-j\sigma} A} \tilde{g}^*(t)\,dt.
\]
If $s>1$, then the H\"older inequality for series with $s' = s/(s-1)$ yields that
\begin{align}
  \label{eq:esssup-p=s-Holder}
	\sum_{n=0}^\infty \biggl(\int_0^{2^{-n\sigma} A} &\tilde{g}^*(t)^s\,dt\biggr)^{1/s} = \sum_{n=0}^\infty \frac{(n+1)^{\alpha/s'}}{(n+1)^{\alpha/s'}} \biggl(\int_0^{2^{-n\sigma} A} \tilde{g}^*(t)^s\,dt\biggr)^{1/s} \\
 & \notag
  \le \biggl(\sum_{n=0}^\infty (n+1)^{\alpha s/s'} \int_0^{2^{-n\sigma} A} \tilde{g}^*(t)^s\,dt\biggr)^{1/s} \biggl( \sum_{n=0}^\infty \frac{1}{(n+1)^{\alpha}} \biggr)^{1/s'},
\end{align}
where the latter series converges since $\alpha > 1$. Next,
\begin{align*}
  \sum_{n=0}^\infty & (n+1)^{\alpha(s-1)} \int_0^{2^{-n\sigma} A} \tilde{g}^*(t)^s\,dt = \sum_{n=0}^\infty \sum_{j=n}^\infty (n+1)^{\alpha(s-1)} \int_{2^{-(j+1)\sigma} A}^{2^{-j\sigma} A} \tilde{g}^*(t)^s\,dt \\
   & \quad = \sum_{j=0}^\infty \sum_{n=0}^j (n+1)^{\alpha(s-1)} \int_{2^{-(j+1)\sigma} A}^{2^{-j\sigma} A} \tilde{g}^*(t)^s\,dt 
   \lesssim \sum_{j=0}^\infty (j+1)^{\alpha s} \int_{2^{-(j+1)\sigma} A}^{2^{-j\sigma} A} \tilde{g}^*(t)^s\,dt.
\end{align*}
We have thus shown for all $s \ge 1$ that 
\[
  \sum_{n=0}^\infty \biggl(\int_0^{2^{-n\sigma} A} \tilde{g}^*(t)^s\,dt\biggr)^{1/s} \lesssim \sum_{j=0}^\infty (j+1)^{\alpha s} \int_{2^{-(j+1)\sigma} A}^{2^{-j\sigma} A} \tilde{g}^*(t)^s\,dt.
\]
We can estimate $(j+1)^{\alpha s} \lesssim (1+\log A/t)^{\alpha s}$ for $t\in (2^{-(j+1)\sigma}A, 2^{-j\sigma}A)$. Therefore,
\[
  \sum_{j=0}^\infty (j+1)^{\alpha s} \int_{2^{-(j+1)\sigma} A}^{2^{-j\sigma} A} \tilde{g}^*(t)^s\,dt \lesssim \sum_{j=0}^\infty \int_{2^{-(j+1)\sigma} A}^{2^{-j\sigma} A} \tilde{g}^*(t)^s \biggl( 1+ \log \frac{A}{t}\biggr)^{\alpha s}\,dt\,.
\]
The decreasing rearrangement $\tilde{g}^*$ is supported in $[0, \meas{\lambda \swidetilde{B}}]$ and hence
\begin{align*}
  \int_0^A & \tilde{g}^*(t)^s \biggl( 1+ \log \frac{A}{t}\biggr)^{\alpha s}\,dt = \int_0^{\meas{\lambda \swidetilde{B}}} \tilde{g}^*(t)^s \biggl( 1+ \log \frac{c_\sigma \meas{\lambda \swidetilde{B}}}{t}\biggr)^{\alpha s}\,dt \\
	& \le (1+\log c_\sigma)^{\alpha s} \int_0^{\meas{\lambda \swidetilde{B}}} \tilde{g}^*(t)^s \biggl( 1+ \log \frac{\meas{\lambda \swidetilde{B}}}{t}\biggr)^{\alpha s}\,dt \approx \| g \chi_{\lambda \widetilde B}\|^s_{L^s(\log L)^{\alpha}(\lambda \widetilde B)}\,.
\end{align*}
Putting all the estimates together, we obtain
\[
  |u(x) - u_{\widetilde{B}}| \lesssim r \frac{\| g \chi_{\lambda \widetilde B}\|_{L^s(\log L)^{\alpha}(\lambda \widetilde B)}}{\meas{\lambda \widetilde{B}}^{1/s}}\,.
\]
The triangle and the $s$-Poincar\'e inequality provide us with the estimate
\begin{align*}
  |u_{\widetilde{B}} - u_B| & \le |u_{\widetilde{B}} - u_{2B}| + |u_{2B} - u_B| \le \fint_{\widetilde B} |u-u_{2B}|\,d\mu + \fint_{B} |u-u_{2B}|\,d\mu \\
  & \lesssim \fint_{2B} |u-u_{2B}|\,d\mu \lesssim r \biggl(\fint_{2\lambda{B}} g^s\,d\mu\biggr)^{1/s} \lesssim r \frac{\| g \chi_{2\lambda B}\|_{L^s(\log L)^{\alpha}(2\lambda B)}}{\meas{2\lambda B}^{1/s}}\,.
\end{align*}
Altogether, we see that
\begin{align*}
  |u(x) - u_B| & \le |u(x) - u_{\widetilde{B}}| + |u_{\widetilde{B}} - u_B| 
  \\
	& \lesssim r \frac{\|g \chi_{2\lambda B}\|_{L^s(\log L)^{\alpha}(2\lambda B)}}{\meas{2\lambda B}^{1/s}} \le \frac{\cemb(2\lambda B) r}{\meas{2\lambda B}^{1/s}} \|g \chi_{2\lambda B}\|_X
	\le c_{B_0} \|g \chi_{2\lambda B}\|_X\,,
\end{align*}
where we applied \eqref{eq:dimension} to estimate $r/\meas{2\lambda B}^{1/s} \le R/(c_s \meas{2\lambda B_0})^{1/s}$. The Lebesgue differentiation theorem, which holds true since $\mu$ is doubling (see Heinonen~\cite[Section~1]{Hei}), yields that a.e.\@ $x\in B$ is a Lebesgue point of $u$, whence this inequality holds for a.e.\@ $x\in B$ and thus for the essential supremum over $B$.
\end{proof}
The technique used in the previous proof for $s>1$ seems to work even if we assume that $X \emb L^s \log L (\log \log L)^{\beta}$ with $\beta > 1/s'$. The main difference in the proof would be to replace $(n+1)^{\alpha/s'}$ in \eqref{eq:esssup-p=s-Holder} by $(n+1)^{1/s'} \log(n+2)^{\beta}$. It also appears to be possible to iterate the logarithm several times raised to suitable powers. Nevertheless, we refrain from properly formulating and proving this claim.

It follows by Talenti~\cite{Tal} that Zygmund--Sobolev functions in $\Rbb^n$, $n\ge2$, are bounded for every $\alpha>1/n'\coloneq 1-1/n$, which indicates that $1$ is not the optimal borderline value for the exponent $\alpha$ in the previous proposition. On the other hand, $1/n'$ is sharp, which can be seen by considering $u(x) = \sqrt{\log \log (e^2/|x|)}$ for $x\in B \coloneq B(0,1) \subset \Rbb^n$, $n\ge 2$. Apparently, $u$ is unbounded even though $u \in N^1L^n(\log L)^\alpha(B) \subset W^1L^n(\log L)^\alpha(B)$ for every $\alpha \in [0, 1/n']$.

We will show in Corollary~\ref{cor:esssup_est_p=dZygViaLor} below that we can in fact obtain the local  essential boundedness of Zygmund--Newtonian functions for all $\alpha > 1/n'$ as a special case of a more general result if $\Pcal$ supports a stronger Poincar\'e inequality.
\begin{df}
The \emph{Lorentz space} $L^{p,1}(\Pcal)$ for $1\le p < \infty$ is the Banach function space that consists of the measurable functions $u: \Pcal \to \overline{\Rbb}$ such that
\[
   \| u \|_{L^{p,1}(\Pcal)} \coloneq \frac{1}{p} \int_0^\infty u^*(t) t^{1/p - 1} \,dt < \infty.
\]
\end{df}
%
% ---------------------------------------------------------
%
The following proposition shows that in the borderline case it suffices that the ($X$-weak) upper gradient lies in the Lorentz $L^{s,1}$ space. A similar claim was proven by Romanov~\cite{Rom} under a highly restrictive assumption that $\mu$ is $s$-Ahlfors regular, i.e., both \eqref{eq:dimension} and \eqref{eq:dimension2} hold with $s=\sigma$. His paper served as an inspiration to use Abel's partial summation formula in the proof below.

The price we have to pay is that a stronger Poincar\'e inequality is needed. Actually, assuming that $\Pcal$ supports an $s$-Poincar\'e inequality is enough if $\Pcal$ is complete. By Keith and Zhong~\cite{KeiZho}, the Poincar\'e inequality is a self-improving property in that case and hence $\Pcal$ supports a $p$-Poincar\'e inequality with some $p<s$.
%
% ---------------------------------------------------------
%
\begin{pro}
\label{pro:esssup_est_p=sLorentz}
Suppose that $\Pcal$ supports a $p$-Poincar\'e inequality with $1\le p < s$ such that \eqref{eq:dimension} is satisfied. Assume that $X \emb L^{s,1}_\loc$. Let $B_0 \subset \Pcal$ be a fixed ball of radius $R>0$. Then, there is a constant $c_{B_0}>0$ such that for every ball $B \subset B_0$ of radius $r\in(0, R)$, we have
\begin{equation}
  \label{eq:CX-esssup_est_p=d_Lor}
  \mbox{$C_X$-}\esssup_{x\in B} |u(x) - u_B| \lesssim \frac{\cemb(2\lambda B) r}{\meas{2\lambda B}^{1/s}} \|g \chi_{2\lambda B}\|_X \le c_{B_0} \|g \chi_{2\lambda B}\|_X
\end{equation}
whenever $g \in X$ is an $X$-weak upper gradient of $u\in \NX$. Moreover, we can estimate $c_{B_0} \approx \cemb(2\lambda B_0) R/\meas{2 \lambda B_0}^{1/s}$.
\end{pro}
\begin{proof}
Let $B_0 = B(y, R) \subset \Pcal$ and $B = B(z, r)\subset B_0$ be arbitrary balls with $r<R$. Then,  $\mbox{$C_X$-}\esssup_{x\in B} |u(x) - u_B| = \esssup_{x\in B} |u(x) - u_B|$ by~\cite[Corollary~6.13]{Mal1} since $|u - u_B| \in DX$.

Let $x \in B$ be a Lebesgue point of $u$ and set $\widetilde{B} = B(x, r)$. Similarly as in the proof of Proposition~\ref{pro:esssup_est_p=s}, we obtain that
\begin{equation}
  \label{eq:esssup-Lorentz-1}
  |u(x) - u_{\widetilde{B}}| \lesssim r \sum_{n=0}^\infty 2^{-n} \biggl(\fint_{2^{-n}\lambda\widetilde{B}} g^p\,d\mu\biggr)^{1/p}\,.
\end{equation}
Let us, for the sake of brevity, write $\tilde{g} = g \chi_{\lambda \widetilde{B}}$. The embedding $L^{p,1} \emb L^p$, whose norm is $1$, and the doubling condition give that
\begin{equation}
 \label{eq:esssup-Lorentz-2}
 \biggl(\fint_{2^{-n}\lambda\widetilde{B}} g^p\,d\mu\biggr)^{1/p} \le \frac{\|g \chi_{2^{-n}\lambda \widetilde{B}}\|_{L^{p,1}(\Pcal)}}{\meass{2^{-n}\lambda \widetilde{B}}^{1/p}} \lesssim \frac{\| \tilde{g}^* \chi_{(0, \meass{2^{-n}\lambda \widetilde{B}})}\|_{L^{p,1}(\Rbb^+)}}{\meass{2^{-n}\widetilde{B}}^{1/p}}\,.
\end{equation}
Let $I_n = \| \tilde{g}^* \chi_{(0, \meass{2^{-n}\lambda \widetilde{B}})}\|_{L^{p,1}(\Rbb^+)}$. Summation by parts allows us to write
\begin{equation}
  \label{eq:AbelsPartSum}
  \sum_{n=0}^N \frac{2^{-n}}{\meass{2^{-n}\widetilde{B}}^{1/p}} I_n = \biggl( \sum_{\vphantom{k}n=0}^{N-1} \sum_{k=0}^n \frac{2^{-k}}{\meass{2^{-k}\widetilde{B}}^{1/p}} (I_n - I_{n+1})\biggr) + \sum_{k=0}^N \frac{2^{-k}}{\meass{2^{-k}\widetilde{B}}^{1/p}} I_N
\end{equation}
for every $N>0$. Inequality \eqref{eq:dimension} yields that $2^{-k}=2^{-k}r/r \lesssim (\meass{2^{-k}\widetilde{B}}/\meass{\widetilde{B}})^{1/s}$ and that $\meass{2^{-N}\lambda \widetilde{B}} \approx \meass{2^{-N} \widetilde{B}}$. Recall also that $1\le p<s$. Then,
\begin{align}
  \notag
  \sum_{k=0}^N \frac{2^{-k}}{\meass{2^{-k}\widetilde{B}}^{1/p}} I_N & \lesssim \sum_{k=0}^N \frac{\meass{2^{-k}\widetilde{B}}^{1/s - 1/p}}{\meass{\widetilde{B}}^{1/s}} \int_0^{\meass{2^{-N}\lambda \widetilde{B}}} \tilde{g}^*(t) t^{1/p - 1}\,dt \\
  \label{eq:APS_est}
	& \lesssim \frac{1}{\meass{\widetilde{B}}^{1/s}} \sum_{k=0}^N \biggl(\frac{\meass{2^{-N}\widetilde{B}}}{\meass{2^{-k}\widetilde{B}}}\biggr)^{1/p - 1/s} \int_0^{\meass{2^{-N}\lambda \widetilde{B}}} \tilde{g}^*(t) t^{1/s - 1}\,dt \\
  & \lesssim \frac{\| \tilde{g}^* \chi_{(0, \meass{2^{-N}\lambda \widetilde{B}})}\|_{L^{s,1}(\Rbb^+)}}{\meass{\widetilde{B}}^{1/s}} \sum_{k=0}^N \bigl(2^{(k-N)\sigma}\bigr)^{1/p-1/s}\,,
	\notag
\end{align}
where the last inequality with some $\sigma \in (0, s]$ follows from \eqref{eq:dimension2}. Due to the absolute continuity of the $L^{s,1}$ norm, we see that the last sum in \eqref{eq:AbelsPartSum} tends to zero as $N\to \infty$. Therefore,
\begin{equation}
  \label{eq:After-APS1}
  \sum_{n=0}^\infty \frac{2^{-n}}{\meass{2^{-n}\widetilde{B}}^{1/p}} I_n = \sum_{\vphantom{k}n=0}^{\infty} \sum_{k=0}^n \frac{2^{-k}}{\meass{2^{-k}\widetilde{B}}^{1/p}} (I_n - I_{n+1}).
\end{equation}
Similarly as in \eqref{eq:APS_est}, we can apply \eqref{eq:dimension} and \eqref{eq:dimension2} to estimate
\begin{align*}
  \sum_{k=0}^n & \frac{2^{-k}}{\meass{2^{-k}\widetilde{B}}^{1/p}} (I_n - I_{n+1}) \\
   & \qquad
   \lesssim \frac{1}{\meass{\widetilde{B}}^{1/s}} \sum_{k=0}^n \biggl(\frac{ \meass{2^{-n}\widetilde{B}}}{\meass{2^{-k}\widetilde{B}}}\biggr)^{1/p - 1/s} \int_{\meass{2^{-n-1}\lambda \widetilde{B}}}^{\meass{2^{-n}\lambda \widetilde{B}}} \tilde{g}^*(t) t^{1/s - 1}\,dt \\
  & \qquad \lesssim \frac{1}{\meass{\widetilde{B}}^{1/s}} \sum_{k=0}^\infty 2^{-k\sigma(1/p-1/s)} \int_{\meass{2^{-n-1}\lambda \widetilde{B}}}^{\meass{2^{-n}\lambda \widetilde{B}}} \tilde{g}^*(t) t^{1/s - 1}\,dt\,.
\end{align*}
Inserting this estimate into \eqref{eq:After-APS1} yields that
\begin{equation}
  \label{eq:esssup-Lorentz-3}
  \sum_{n=0}^\infty \frac{2^{-n}}{\meass{2^{-n}\widetilde{B}}^{1/p}} I_n \lesssim \frac{1}{\meass{\widetilde{B}}^{1/s}} \int_0^{\meass{\lambda \widetilde{B}}} \tilde{g}^*(t) t^{1/s-1}\,dt = \frac{\|\tilde{g}^*\|_{L^{s,1}(\Rbb^+)}}{\meass{\widetilde{B}}^{1/s}}
\end{equation}
Combining \eqref{eq:esssup-Lorentz-1}, \eqref{eq:esssup-Lorentz-2}, and \eqref{eq:esssup-Lorentz-3} results in
\[
  |u(x) - u_{\widetilde{B}}| \lesssim \frac{r}{\meass{\widetilde{B}}^{1/s}} \|\tilde{g}\|_{L^{s,1}(\Pcal)} \approx \frac{r}{\meass{2\lambda \widetilde{B}}^{1/s}} \|g \chi_{\lambda \widetilde{B}}\|_{L^{s,1}(\Pcal)}\,.
\]
The triangle, the $p$-Poincar\'e, and the H\"older inequality, as well as the embedding $L^{s,1} \emb L^s$ provide us with the estimate
\begin{align*}
  |u_{\widetilde{B}} &- u_B|  \le |u_{\widetilde{B}} - u_{2B}| + |u_{2B} - u_B| \le \fint_{\widetilde B} |u-u_{2B}|\,d\mu + \fint_{B} |u-u_{2B}|\,d\mu \\
   & \lesssim \fint_{2B} |u-u_{2B}|\,d\mu \lesssim r \biggl(\fint_{2\lambda{B}} g^p\,d\mu\biggr)^{1/p} \le r \biggl(\fint_{2\lambda{B}} g^s\,d\mu\biggr)^{1/s}
	\lesssim r \frac{\| g \chi_{2\lambda B}\|_{L^{s,1}(\Pcal)}}{\meas{2\lambda B}^{1/s}}\,.
\end{align*}
Altogether, we see that
\begin{align*}
  |u(x) - u_B| & \le |u(x) - u_{\widetilde{B}}| + |u_{\widetilde{B}} - u_B| 
  \\
	& \lesssim r \frac{\| g \chi_{2\lambda B}\|_{L^{s,1}(\Pcal)}}{\meas{2\lambda B}^{1/s}} \le \frac{\cemb(2\lambda B) r}{\meas{2\lambda B}^{1/s}} \|g \chi_{2\lambda B}\|_X
	\le c_{B_0} \|g \chi_{2\lambda B}\|_X\,,
\end{align*}
where we applied \eqref{eq:dimension} to estimate $r/\meas{2\lambda B}^{1/s} \le R/(c_s \meas{2\lambda B_0})^{1/s}$. The Lebesgue differentiation theorem (see Heinonen~\cite[Section~1]{Hei}) now yields that a.e.\@ $x\in B$ is a Lebesgue point of $u$, whence this inequality holds for a.e.\@ $x\in B$ and thus for the essential supremum over $B$.
\end{proof}
%
% ---------------------------------------------------------
%
Edmunds, Kerman and Pick~\cite{EdmKerPic} have discussed the optimal Sobolev embeddings $W^m X(\Omega) \emb Y(\Omega)$ for bounded domains $\Omega \subset \Rbb^n$. It follows from~\cite[Theorem~6.5]{EdmKerPic} that $X=L^{n,1}(\Omega)$ is the largest \ri space such that $W^1 X(\Omega) \emb L^\infty(\Omega)$. Therefore, the result we have just obtained is sharp when considering \ri spaces as the base function spaces $X$ that $\NX$ are built upon.

As mentioned earlier, we also recover an analogue of the result of Talenti~\cite{Tal} on local essential boundedness of Newtonian functions based on the Zygmund space $X=L^s(\log L)^\alpha$ with $\alpha>1/s'=1-1/s$. Namely, it follows from the embedding $L^s(\log L)^\alpha_\loc \emb L^{s,1}_\loc$, which is shown next.
%
% ---------------------------------------------------------
%
\begin{lem}
\label{lem:LpLog-emb}
Let $1<p<\infty$ and suppose that $E\subset \Pcal$ is a measurable set of finite positive measure. Then, $L^p(\log L)^\alpha(E) \emb L^{p,1}(E)$ if and only if $\alpha > 1-1/p$.
\end{lem}
\begin{proof}
Let $a = \meas{E}$. We want to show that
\[
  \int_0^a f(t) t^{1/p -1} \,dt \lesssim \biggl( \int_0^a f(t)^p \biggl( 1 + \log \frac{a}{t} \biggr)^{\alpha p} \,dt\biggr)^{1/p}
\]
for every non-negative decreasing $f\in \Mcal(\Rbb^+, \lambda^1)$. According to Stepanov~\cite[Proposition 1]{Ste}, this inequality holds true if and only if
\[
  b \coloneq \int_0^a \biggl(\frac{\int_0^t \tau^{1/p - 1} \,d\tau}{\int_0^t (1+\log (a/\tau))^{\alpha p}\,d\tau}\biggr)^{1/(p-1)} t^{1/p -1}\,dt < \infty\,.
\]
We have the rough estimate $\int_0^t (1+\log (a/\tau))^{\alpha p}\,d\tau \ge t (1+\log (a/t))^{\alpha p}$. Conversely,
\begin{align*}
  \int_0^t \biggl( 1 + \log \frac a\tau\biggr)^{\alpha p}\,d\tau & = \sum_{n=0}^\infty \int_{2^{-n-1}t}^{2^{-n}t} \biggl( 1 + \log \frac a\tau\biggr)^{\alpha p}\,d\tau \le \sum_{n=0}^\infty 2^{-n-1}t \biggl( 1 + \log \frac{a}{2^{-n-1}t}\biggr)^{\alpha p} \\
	& \le \sum_{n=0}^\infty 2^{-n-1} t (1+ \log 2^{n+1})^{\alpha p} \biggl( 1 + \log \frac{a}{t}\biggr)^{\alpha p} \lesssim t \biggl( 1 + \log \frac{a}{t}\biggr)^{\alpha p}\,.
\end{align*}
Therefore, we obtain that
\[
  b \approx \int_0^a \biggl(\frac{t^{1/p}}{t (1+\log (a/t))^{\alpha p}}\biggr)^{1/(p-1)} t^{1/p -1}\,dt \approx \int_0^a \frac{dt}{t (1+ \log (a/t))^{\alpha p/(p-1)}}\,.
\]
The integral on the right-hand side converges if and only if $\alpha > (p-1)/p$.
\end{proof}
\begin{cor}
\label{cor:esssup_est_p=dZygViaLor}
Assume that $\Pcal$ supports a $p$-Poincar\'e inequality with $1 \le p < s$ such that \eqref{eq:dimension} is satisfied. Suppose also that $X\emb L^s(\log L)^\alpha_\loc$ with some $\alpha > 1-1/s$. Let $B_0 \subset \Pcal$ be a fixed ball of radius $R>0$. Then, there is a constant $c_{B_0}>0$ such that for every ball $B \subset B_0$ of radius $r\in(0, R)$, we have
\[
  \mbox{$C_X$-}\esssup_{x\in B} |u(x) - u_B| \lesssim \frac{\cemb(2\lambda B) r}{\meas{2\lambda B}^{1/s}} \|g \chi_{2\lambda B}\|_X \le c_{B_0} \|g \chi_{2\lambda B}\|_X
\]
whenever $g \in X$ is an $X$-weak upper gradient of $u\in \NX$. Moreover, we can estimate $c_{B_0} \approx \cemb(2\lambda B_0) R/\meas{2 \lambda B_0}^{1/s}$.
\end{cor}
%
% ---------------------------------------------------------
%
Similarly as in Proposition~\ref{pro:esssup_est_p=sLorentz}, it suffices to assume that $\Pcal$ supports an $s$-Poincar\'e inequality if $\Pcal$ is complete.
\begin{proof}
By Lemma~\ref{lem:LpLog-emb}, we see that $X \emb L^{s,1}_\loc$. Then, the desired claim follows directly from Proposition~\ref{pro:esssup_est_p=sLorentz}.
\end{proof}
%
%
%  SECTION 7: CONTINUITY OF NEWTONIAN FUNCTIONS
%
%
\section{Continuity of Newtonian functions}
\label{sec:continuity}
In Sobolev spaces in $\Rbb^n$, one may deduce that there exist continuous representatives if the $L^\infty$ norm of a function is (locally) controlled by the quasi-norm of its gradient. A similar result, yet somewhat stronger, can be obtained in Newtonian spaces as well. Namely, it suffices in the Newtonian case to redefine a function on a set of capacity zero to obtain the continuous representative. If the metric measure space is in addition locally compact, then all representatives are continuous.

We assume in this (as well as in the previous) section that $\mu$ is a \emph{doubling measure}. Moreover, $\mu$ is non-atomic since $\Pcal$ is connected, which follows from the Poincar\'e inequalities we will assume.

The following theorem is a refinement of Haj\l{}asz and Koskela~\cite[Theorem~5.1]{HajKos}, where a $p$-Poincar\'e inequality was used to show that there exist $(1-s/p)$-H\"older continuous representatives (with equality a.e.\@) whenever the upper gradient lies in $L^p$ and $p>s$.
The case when the degree of summability of an upper gradient 
is essentially equal to $s$ needs to be discussed using a finer scale of function spaces.

In $\Rbb^n$, Kauhanen, Koskela, and Mal\'y~\cite{KauKosMal} have shown that it suffices that the gradient lies in the Lorentz space $L^{n,1}_\loc$ to conclude that there are continuous representatives. Romanov~\cite{Rom} extended this result to Sobolev-type spaces on complete metric measure spaces with an Ahlfors $s$-regular measure (i.e., $\meas{B(x,r)} \approx r^s$ for $r<2\diam \Pcal$). The Ahlfors $s$-regularity of the measure is a very strong requirement that fails even in $(\Rbb^n, w(x)\,dx)$ unless $w(x)\approx 1$. Besides, he did not work with Newtonian spaces as such, but solely with Poincar\'e inequalities. His result can be recovered as a special case of \ref{it:p=s-cont-Lorentz} in the following theorem.
\begin{thm}
\label{thm:noncompl-cont}
Suppose that $\Pcal$ supports a $p$-Poincar\'e inequality and let $s$ be given by~\eqref{eq:dimension}. Suppose that one of the following sets of assumptions is satisfied:
\begin{enumerate}
	\item \label{it:noncompl-cont-Holder} $p> s$ and $X \emb L^p_\loc\fcrim$;
	\item $1=p=s$ and $X \emb L^s (\log L)^\alpha_\loc\fcrim$ for some $\alpha \ge 1$;
	\item $1<p=s$ and $X \emb L^s (\log L)^\alpha_\loc\fcrim$ for some $\alpha > 1$;
	\item $1\le p < s$ and $X \emb L^{s,1}_\loc$\,.\label{it:p=s-cont-Lorentz}
\end{enumerate}
Then, for every function $u \in \NX$, there is $v \in \NX \cap \Ccal(\Pcal)$ such that $u=v$ $C_X$-quasi-everywhere. Moreover, $v$ is locally $(1-s/p)$-H\"older continuous in the case \ref{it:noncompl-cont-Holder}.
\end{thm}
\begin{proof}
Let us exhaust $\Pcal = \bigcup_{n=1}^\infty B_n$, where $B_n = B(x_0, n)$ for an arbitrary point $x_0 \in \Pcal$. Next, we will find an exceptional set $E$ with $C_X(E) = 0$, where we lack control over the oscillation of $u \in \NX$.
Let $D=\{z_i \in \Pcal: i\in\Nbb\}$ be a dense subset of $\Pcal$ and let
\[
  E = \bigcup_{i=1\vphantom{\Qbb^+}}^\infty \bigcup_{\:q\in\Qbb^+} \biggl\{x \in B(z_i, q): |u(x) - u_{B(z_i, q)}| > \mathop{\mbox{$C_X$-}\esssup}_{ w \in B(z_i, q)} |u(w) - u_{B(z_i, q)}|\biggr\}\,.
\]
Since $E$ is a countable union of sets of capacity zero, it has capacity zero as well.

Let us now fix a ball $\widetilde{B} \coloneq 4 B_n = B_{4n}$ for an arbitrary $n \in \Nbb$. For every pair of points~$x,y \in B_n\setminus E$, we can find $z \in D$ and $r \in \Qbb^+$ such that $\dd(x,y)/2 \le r \le 2\dd(x,y)$ and $x,y \in B \coloneq B(z, r) \subset \widetilde{B}$.

In the case \ref{it:noncompl-cont-Holder}, we use \eqref{eq:CX-esssup_est_p>s} to obtain that
\begin{align*}
  |u(x) - u(y)| & \le 2 \mbox{$C_X$-}\esssup_{w \in B} |u(w) - u_{B}| \le c_{\widetilde{B}}\, r^{1-s/p} \|g_u \chi_{2 \lambda B}\|_X \\
	& \le c_{\widetilde{B}} (2\dd(x,y))^{1-s/p} \|g_u\|_X\,.
\end{align*}
Therefore, $u|_{B_n \setminus E}$ is $(1-s/p)$-H\"older continuous and hence uniformly continuous. Since $C_X(E \cap B_n) = 0$, every point of $E \cap B_n$ is an accumulation point of $B_n \setminus E$, whence there is a unique continuation $v_n\in \Ccal(B_n)$ of $u|_{B_n \setminus E}$. Moreover, $v_n$ retains the H\"older continuity.

Let us now focus on the remaining three cases.
Respective to the assumptions, let $Y=Y(\Pcal)$ be either the Zygmund space $L^s(\log L)^\alpha(\Pcal)$ with norm given by \eqref{eq:LpLog-eqnorm}, or the Lorentz space $L^{s,1}(\Pcal)$. Since $|u - u_{B}| \in DX \subset DY_\loc$, we obtain from~\cite[Corollary~6.13]{Mal1} that
\[
 C_X\mbox{-}\esssup_{w \in B} |u(w) - u_{B}| = \esssup_{w \in B} |u(w) - u_{B}| = C_Y\mbox{-}\esssup_{w \in B} |u(w) - u_{B}|\,.
\]
Applying \eqref{eq:CX-esssup_est_p=d_Zyg} or \eqref{eq:CX-esssup_est_p=d_Lor} for the function space $Y$, we can find $c_{\widetilde{B}}>0$ such that
\begin{align*}
  |u(x) - u(y)| & \le 2 C_Y\mbox{-}\esssup_{w \in B} |u(w) - u_{B}| \le c_{\widetilde{B}} \|g_u \chi_{2\lambda B}\|_Y \le c_{\widetilde{B}} \|g_u \chi_{E_B}\|_Y\,,
\end{align*}
where $E_B \subset 2\lambda \widetilde{B}$ with $\meas{E_{B}} = \meas{2\lambda B}$ and $g_u(w) \ge g_u(v)$ for all $w \in E_B$ and $v \in 2\lambda \widetilde{B} \setminus E_B$. Note that the set $E_B$ does not depend on the exact choice of $B$, but merely on the measure of $2\lambda B$.
Given an $\eps>0$, we can find $a>0$ such that  $c_{\widetilde{B}} \|g_u \chi_{E_B}\|_Y < \eps$ whenever $\meas{E_B} < a$ since $Y$ has absolutely continuous norm.
As $\Pcal$ is connected, we have $\meas{2\lambda B} \le C_{2\lambda \swidetilde{B}} \dd(x,y)^\sigma$, where $\sigma \in (0, s]$ is from \eqref{eq:dimension2}.
If $\dd(x,y) < (a / C_{2\lambda\swidetilde{B}})^{1/\sigma}$, then $\meas{E_{B}} = \meas{2\lambda B} < a$. Thus, $c_{\widetilde{B}} \|g_u \chi_{E_B}\|_Y < \eps$ and hence $|u(x) - u(y)| < \eps$.

This way, we have just shown uniform continuity of $u|_{B_n \setminus E}$. Thus, there is a unique continuation $v_n\in \Ccal(B_n)$ of $u|_{B_n \setminus E}$. 

We have thus proven that in all the cases \ref{it:noncompl-cont-Holder}--\ref{it:p=s-cont-Lorentz} there is a unique continuous extension $v_n$ of $u|_{B_n \setminus E}$ for every ball $B_n$, $n\in\Nbb$. Now, we define $v$ on $\Pcal$ by setting $v(x)=v_n(x)$ whenever $x\in B_n$. Then, $v\in\Ccal(\Pcal)$ and $v=u$ outside of $E$, i.e., $C_X$-quasi-everywhere. Furthermore, $v\in\Ccal^{0,1-s/p}_{\loc}(\Pcal)$ in the case \ref{it:noncompl-cont-Holder}.
\end{proof}
Several qualitative properties of the Sobolev capacity have been discussed in Section~\ref{sec:quasicontinuity}, where one of the crucial assumptions was density of continuous functions. Now, we have shown that under certain hypotheses, all Newtonian functions have continuous representatives, whence the continuous functions are dense. Thus, we may formulate the following corollary.
\begin{cor}
Suppose that $\Pcal$ is locally compact and $X$ has the Vitali--Carath\'eodory property \eqref{eq:VitaliCaratheodory}. In particular, it suffices to assume that $\chi_B \in X$ for every bounded set $B\subset \Pcal$ and it satisfies \ref{df:AC}. Suppose further that the hypotheses of Theorem~\ref{thm:noncompl-cont} are satisfied. Then, $\widetilde{C}_{X,r}$ is an outer capacity if $X$ is $r$-normed. In particular, $C_X$ is an outer capacity if $X$ is normed.
\end{cor}
\begin{proof}
Proposition~\ref{pro:VitaliCarath} yields that $X$ has the Vitali--Carath\'eodory property under the given assumptions on $\chi_B$ for bounded sets $B\subset \Pcal$. By Theorem~\ref{thm:noncompl-cont}, every Newtonian function has a continuous representative. Hence, continuous functions are dense in $\NX$. By Corollary~\ref{cor:newt.fcns-are-qcont}, every Newtonian function is quasi-continuous. Finally, it follows from Proposition~\ref{pro:C_X-outer} that $\widetilde{C}_{X,r}$ and $C_X$ are outer capacities.
\end{proof}
We are ready to apply Propositions~\ref{pro:esssup_est_p>s}, \ref{pro:esssup_est_p=s}, and~\ref{pro:esssup_est_p=sLorentz} to find a lower bound for the Sobolev capacity of a subset of a ball, in terms of the measure and radius of the ball, if we know beforehand that the set has non-zero capacity.
%
% ---------------------------------------------------------
%
\begin{pro}
\label{pro:cap_estimate_below}
Suppose that $\Pcal$ supports a $p$-Poincar\'e inequality and let $s$ be given by \eqref{eq:dimension}. Suppose that one of the following sets of assumptions is satisfied:
\begin{enumerate}
	\item $1\le s<p\le q$ and $X \emb L^q_\loc$;
	\item $1=p=s\eqcolon q$ and $X \emb L^s (\log L)^\alpha_\loc\fcrim$ for some $\alpha \ge 1$;
	\item $1<p=s\eqcolon q$ and $X \emb L^s (\log L)^\alpha_\loc\fcrim$ for some $\alpha > 1$;
	\item $1\le p < s \eqcolon q$ and $X \emb L^{s,1}_\loc$\,.
\end{enumerate}
Let $B \subset \Pcal$ be a ball with radius $r>0$. Then, for every $E\subset B$ with $C_X(E)>0$, we can estimate
\[
  C_X(E) \gtrsim \frac{\meas{2\lambda B}^{1/q}}{\cemb(2\lambda B) (r+1)}.
\]
In particular, this estimate holds if $\meas{E}>0$.
\end{pro}
%
% ---------------------------------------------------------
%
\begin{proof}
Let $u \in \NX$ be such that $\chi_E \le u \le 1$ in $\Pcal$ and let $g_u \in X$ be a minimal $X$-weak upper gradient of $u$. Propositions~\ref{pro:esssup_est_p>s}, \ref{pro:esssup_est_p=s}, and~\ref{pro:esssup_est_p=sLorentz} and the H\"older inequality then yield
\begin{align*}
  1 & \le | 1 - u_B | + u_B = C_X\mbox{-}\esssup_{x\in B} |u(x) - u_B| + \fint_B u\,d\mu \\ 
  & \lesssim \frac{\cemb(2\lambda B) r}{\meas{2\lambda B}^{1/q}} \|g_u \chi_{2\lambda B}\|_X  +  \frac{\|u \chi_{B}\|_{L^q}}{\meas{B}^{1/q}} \lesssim \frac{\cemb(2\lambda B)(r+1) \|u\|_\NX}{\meas{2\lambda B}^{1/q}}.
\end{align*}
Taking infimum over all such functions $u\in \NX$, we obtain
\[
  C_X(E) = \inf \|u\|_\NX \gtrsim \frac{\meas{2\lambda B}^{1/q}}{\cemb(2\lambda B) (r+1)}\,.
  \qedhere
\]
\end{proof}
%
% ---------------------------------------------------------
%
If $X$ is an \ri space, then it is possible to find an estimate of the capacity expressed using the fundamental function of $X$, provided that the integral means in $L^p$ can be suitably rescaled to the norm means in $X$.
\begin{cor}
\label{cor:cap_estimate_below_p>s}
Assume that $\Pcal$ supports a $p$-Poincar\'e inequality with $p > s$ such that~\eqref{eq:dimension} is satisfied. Let $B \subset \Pcal$ be a ball with radius $r>0$. Suppose further that $X$ is an \ri space with fundamental function $\phi$ such that $c_\phi(\Pcal) < \infty$, where $c_\phi$ is defined by~\eqref{eq:MX-Lp-emb_norm}, see Corollary~\ref{cor:esssup_est_p>s}. Then, for every $E\subset B$ with $C_X(E)>0$, we can estimate
\[
  C_X(E) \gtrsim \frac{\phi(\meas{2\lambda B})}{r+1}\,.
\]
In particular, this estimate holds if $\meas{E}>0$.
\end{cor}
\begin{proof}
In principle, it was shown in the proof of Corollary~\ref{cor:esssup_est_p>s} that 
\[
  \cemb(2\lambda B) \le c_\phi(\Pcal) \meas{2\lambda B}^{1/p} / \phi(\meas{2\lambda B}).
\]
The desired result follows from Proposition~\ref{pro:cap_estimate_below}.
\end{proof}
%
% ---------------------------------------------------------
%
Recall that the natural equivalence classes in $\NX$ are given by equality outside of sets of capacity zero. Therefore, in order to be able to prove that all Newtonian functions in a locally compact doubling Poincar\'e space are continuous if the summability of the upper gradients (in terms of $\|\cdot\|_X$) is sufficiently high, we need to show that singletons have positive capacity.  To that end, we will apply the outer regularity of the capacity on zero sets.
\begin{pro}
\label{pro:singleton-cap}
Assume that $\Pcal$ is locally compact and supports a $p$-Poincar\'e inequality and let $s$ be given by \eqref{eq:dimension}. Suppose that one of the following sets of assumptions is satisfied:
\begin{enumerate}
	\item $1\le s<p\le q$ and $X \emb L^q_\loc$;
	\item $1=p=s\eqcolon q$ and $X \emb L^s (\log L)^\alpha_\loc\fcrim$ for some $\alpha \ge 1$;
	\item $1<p=s\eqcolon q$ and $X \emb L^s (\log L)^\alpha_\loc\fcrim$ for some $\alpha > 1$;
	\item $1\le p < s \eqcolon q$ and $X \emb L^{s,1}_\loc$\,.
\end{enumerate}
Then, for every ball $B \subset \Pcal$ of radius $r>0$ and every $x \in B$, we have $C_X(\{x\}) \gtrsim \meas{B}^{1/q}/(r+1) > 0$.
\end{pro}
\begin{proof}
Let $B \subset \Pcal$ be fixed. Let $Y=Y(B)$ be either $L^q(B)$, or $L^s(\log L)^\alpha(B)$, or $L^{s,1}(B)$ as in the proposition's hypotheses. Then, $Y$ has an absolutely continuous norm and satisfies \ref{df:BFS.finmeasfinnorm}, whence $Y$ has the Vitali--Carath\'eodory property by Proposition~\ref{pro:VitaliCarath}. 
Let $x\in B$ and suppose for a moment that $C_Y(\{x\}) = 0$. Then, Proposition~\ref{pro:Cap-out-regular_0} yields that $C_Y(\{x\}) = \inf_{G\ni x} C_Y(G)$, where $G$ is open. Such a set $G$ has positive measure and hence $C_Y(G) > 0$. Thus, we can estimate $C_Y(G)\ge c(B) > 0$ whenever $G \subset B$ by Proposition~\ref{pro:cap_estimate_below}. Hence, $C_Y(\{x\}) \ge c(B) > 0$, which contradicts the assumption $C_Y(\{x\}) = 0$.
Therefore, $0< C_Y(\{x\}) \lesssim C_X(\{x\})$ and the claimed estimate follows from Proposition~\ref{pro:cap_estimate_below}.
\end{proof}
%
% ---------------------------------------------------------
%
In view of the previous proposition, we see that the Newtonian functions considered in Propositions~\ref{pro:esssup_est_p>s}, \ref{pro:esssup_est_p=s}, and~\ref{pro:esssup_est_p=sLorentz} are not only $C_X$-essentially bounded on all balls, but bounded everywhere in the respective balls, provided that $\Pcal$ is locally compact.

Similarly, we will next show that Newtonian functions not only have continuous representatives, but in fact are continuous. Thus, the claim is stronger than its analogue for Sobolev functions in $\Rbb^n$.
\begin{thm}
\label{thm:NX_cont}
Under the assumptions of Proposition~\ref{pro:singleton-cap}, every $u\in\NX$ is continuous.
\end{thm}
\begin{proof}
Let $u\in \NX$. By Theorem~\ref{thm:noncompl-cont}, there is $v \in \NX \cap \Ccal(\Pcal)$ such that $u=v$ q.e. In other words, the set $E=\{x\in \Pcal: u(x)\neq v(x)\}$ has zero capacity. According to Proposition~\ref{pro:singleton-cap}, $E$ cannot contain any single point $x\in \Pcal$ as it would have positive capacity then. Therefore, $u=v$ everywhere in $\Pcal$, whence $u\in \Ccal(\Pcal)$.
\end{proof}
%
%
%  APPENDIX: WUG CALCULUS
%
%
\section*{Appendix: Calculus for weak upper gradients}
\setcounter{section}{1}
\setcounter{thm}{0}
\renewcommand{\thesection}{\Alph{section}}
Throughout the paper, several tools for working with weak upper gradients have been needed. Note that none of the results in this section requires the measure to be doubling.
\begin{thm}[Product rule]
\label{thm:product_rule}
Let $u,v : \Pcal \to \overline{\Rbb}$ be measurable. Assume that there are measurable $g,h\ge 0$ such that 
$u\circ \gamma, v\circ \gamma \in \AC([0, l_\gamma])$ with 
\[
  |(u\circ \gamma)'(t)| \le g(\gamma(t)) \quad \mbox{and} \quad |(v\circ \gamma)'(t)| \le h(\gamma(t))\quad\mbox{for a.e.\@ $t\in (0, l_\gamma)$}
\]
for $\Mod_X$-a.e.\@ curve $\gamma: [0, l_\gamma]\to \Pcal$.
In particular, it suffices to assume that $g, h \in X$ are $X$-weak upper gradients of $u,v\in DX$, respectively.
Then, $|u|h + |v|g$ is an $X$-weak upper gradient of $uv$.
\end{thm}
\begin{proof}
If $g,h\in X$ are $X$-weak upper gradients of $u,v \in DX$, then the hypotheses are satisfied by~\cite[Theorem~6.7 and Lemma~6.8]{Mal1}.

Now, let $\gamma: [0, l_\gamma] \to \Pcal$ be a curve for which the theorem's conditions on $u\circ \gamma, v\circ\gamma$ and $g,h$ are satisfied. Let $w = uv$. Then, $w\circ \gamma \in \AC([0, l_\gamma])$ by~\cite[Lemma~1.58]{BjoBjo} and
\begin{align*}
  |(w\circ \gamma)'(t)| & = | u(\gamma(t))(v\circ \gamma)'(t) + v(\gamma(t))(u\circ \gamma)'(t)| \\
	& \le \left|u(\gamma(t))\right| h(\gamma(t)) + \left|v(\gamma(t))\right|g(\gamma(t)) \qquad\mbox{\quad for a.e.\@ $t\in (0, l_\gamma)$.}
\end{align*}
Finally, $|u|h+|v|g$ is measurable and hence an $X$-weak upper gradient of $w$ by~\cite[Lemma~6.8]{Mal1}.
\end{proof}
Observe that $|u|h + |v|g$ need not be a minimal $X$-weak upper gradient of $uv$ even if $g$ and $h$ were minimal $X$-weak upper gradients of $u$ and $v$, respectively. For example, suppose that $X = L^p(0,1)$ with $p\ge 1$ and let $u(t)=g(t)=e^t$ and $v(t)=h(t)=e^{-t}$ for $t\in(0,1)$. Then, $uv \equiv 1$, whence $0$ is an upper gradient of $uv$, but $|u|h + |v|g \equiv 2$.
\begin{thm}[Chain rule]
\label{thm:chain_rule}
Suppose that $g_u \in X$ is a minimal $X$-weak upper gradient of $u\in\NX$. Let $\phi: \overline{\Rbb} \to \overline{\Rbb}$ be locally Lipschitz on $\Rbb$. Then, $|\phi'\circ u|g_u$ is an $X$-weak upper gradient of $\phi \circ u$, where $\phi'\circ u \coloneq 0$ wherever undefined. Moreover, if $|\phi'\circ u|g_u \in X$, then it is a minimal $X$-weak upper gradient of $\phi\circ u$.
\end{thm}
\begin{proof}
It follows from~\cite[Theorem~6.7 and Lemma~6.8]{Mal1} that we have for $\Mod_X$-a.e.\@ curve $\gamma : [0, l_\gamma]\to\Pcal$ that $u\circ \gamma$ is absolutely continuous and
\[
  |(u\circ \gamma)'(t)| \le g_u(\gamma(t)) \quad\mbox{for a.e.\@ $t\in (0, l_\gamma)$}.
\]
Let $I=u\circ \gamma([0, l_\gamma])$. Then, $I$ is a compact (possibly degenerate) interval in $\Rbb$. Thus, $\phi|_I$ is Lipschitz continuous, whence $\phi\circ(u\circ \gamma) \in \AC([0, l_\gamma])$ by~\cite[Lemma~1.58]{BjoBjo} and $(\phi \circ u\circ \gamma)'(t)$ exists for a.e.\@ $t\in (0, l_\gamma)$. We may apply the chain rule of Mal\'{y} and Ziemer~\cite[Theorem 1.74\,(i)]{MalZie} to obtain that
\begin{equation}
   \label{eq:chain_r}
   |(\phi\circ u \circ \gamma)'(t)| = |\phi'(u(\gamma(t)))|\, |(u \circ \gamma)'(t)| \le |\phi'(u(\gamma(t)))|  g_u(\gamma(t))
\end{equation}
for a.e.\@ $t\in(0, l_\gamma)$ provided that we interpret the expression in the middle as $0$ whenever $(u \circ \gamma)'(t)=0$ even if $\phi'(u(\gamma(t)))$ is not defined. Note that the set, where $(u \circ \gamma)'(t) \neq 0$ and $\phi'(u(\gamma(t)))$ does not exist, has zero measure in $[0, l_\gamma]$. By~\cite[Lemma~6.8]{Mal1}, we see that $|\phi'\circ u|g_u$ is an $X$-weak upper gradient of $\phi\circ u$.

Suppose now that $|\phi'\circ u|g_u \in X$. Then, $\phi\circ u \in DX$ and there exists a minimal $X$-weak upper gradient $g_{\phi\circ u} \in X$ of $\phi\circ u$. From \eqref{eq:chain_r} and~\cite[Lemma~6.8]{Mal1}, we also have for a.e.\@ $t\in(0, l_\gamma)$ that
\[
  |(u \circ \gamma)'(t)| \le \frac{|(\phi\circ u \circ \gamma)'(t)|}{|\phi'(u(\gamma(t)))|} \le \frac{g_{\phi\circ u}(\gamma(t))}{|\phi'(u(\gamma(t)))|},
\]
where the fractions are interpreted as $\infty$ whenever $\phi'(u(\gamma(t)))=0$.
Applying~\cite[Lemma~6.8]{Mal1} again, we see that $g_{\phi\circ u} / |\phi'\circ u|$ is an $X$-weak upper gradient of $u$.

Minimality of $g_{\phi\circ u}$ yields that $g_{\phi\circ u} \le |\phi'\circ u|g_u$ a.e. Similarly, minimality of $g_u$ then yields that $g_u \le g_{\phi\circ u} / |\phi'\circ u|$ a.e. Hence, $g_{\phi\circ u} = g_u |\phi'\circ u|$ a.e.\@ in $\Pcal$.
\end{proof}
%
% ---------------------------------------------------------
%
\begin{pro}
Let $g_u, g_v \in X$ be minimal $X$-weak upper gradients of $u, v\in DX$, respectively. Then, $g\coloneq g_u \chi_{\{u>v\}} + g_v \chi_{\{u\le v\}}$ is a minimal $X$-weak upper gradient of $w\coloneq \max\{u, v\}$.
\end{pro}
\begin{proof}
For $\Mod_X$-a.e.\@ curve $\gamma: [0, l_\gamma] \to \Pcal$, the functions $u\circ \gamma$ and $v\circ \gamma$ are absolutely continuous and
\[
  |(u\circ \gamma)'(t)| \le g_u(\gamma(t)) \quad \mbox{and} \quad |(v\circ \gamma)'(t)| \le g_v(\gamma(t))\quad\mbox{for a.e.\@ $t\in (0, l_\gamma)$}
\]
by~\cite[Theorem~6.7 and Lemma~6.8]{Mal1}. Then, $w\circ \gamma \in \AC([0, l_\gamma])$ by~\cite[Lemma~1.58]{BjoBjo}.

The set $\gamma^{-1}(\{x\in\Pcal: u(x) > v(x) \})$ is open in $[0, l_\gamma]$, whence $|(w\circ \gamma)'(t)| \le g_u(\gamma(t))$ for a.e.\@ $t\in \gamma^{-1}(\{x\in\Pcal: u(x) > v(x) \})$. Similarly, $|(w\circ \gamma)'(t)| \le g_v(\gamma(t))$ for a.e.\@ $t\in \gamma^{-1}(\{x\in\Pcal: u(x) < v(x) \})$.

It remains to discuss what happens on the set where $u=v$. Let $t\in \gamma^{-1}(\{x\in\Pcal: u(x) = v(x) \})$ be chosen such that both $(u\circ \gamma)'(t)$ and $(v\circ \gamma)'(t)$ exist. Then, either $(u\circ \gamma)'(t) = (v\circ \gamma)'(t) = (w\circ \gamma)'(t)$, or $(u\circ \gamma)'(t)\neq (v\circ \gamma)'(t)$ in which case $(w\circ \gamma)'(t)$ does not exist.

Therefore, $|(w\circ \gamma)'(t)| \le g(\gamma(t))$ for a.e.\@ $t\in (0, l_\gamma)$. By applying~\cite[Lemma~6.8]{Mal1} again, $g$ is an $X$-weak upper gradient of $w$, whence $w\in DX$ as $g\in X$.

There exists a minimal $X$-weak upper gradient $g_w \in X$ of $w$ by~\cite[Theorem~4.6]{Mal2}. Since $-u = \max\{-u, -w\}$, the previous part of the proof yields that $g_u \le g_u \chi_{\{u<w\}} + g_w \chi_{\{u \ge w\}}$ a.e.\@ due to the minimality of $g_u$. Hence, $g_u \le g_w \le g = g_u$ a.e.\@ on $\{x\in\Pcal: u(x) > v(x)\}$. Similarly, $g_v \le g_w \le g = g_v$ a.e.\@ on $\{x\in\Pcal: u(x) \le v(x)\}$. We have shown that $g_w = g$ a.e., so $g$ is a minimal $X$-weak upper gradient of $w$.
\end{proof}
%
% ---------------------------------------------------------
%
The following claim shows that minimal weak upper gradients depend only on the local behavior of Newtonian functions.
%
% ---------------------------------------------------------
%
\begin{cor}
\label{cor:u=v_gu=gv}
Let $g_u, g_v \in X$ be minimal $X$-weak upper gradients of $u, v\in DX$, respectively. Then,
$g_u = g_v$ a.e.\@ on $E\coloneq\{x\in\Pcal: u(x) = v(x)\}$. In particular, we obtain that $g_u = 0$ a.e.\@ on $\{x\in\Pcal: u(x) = c\}$ for any constant $c\in\Rbb$.
\end{cor}
\begin{proof}
Let $w= \max\{u, v\}$. Then, $g_w = g_u \chi_{\{u>v\}} + g_v \chi_{\{u \le v\}}$ is a minimal $X$-weak upper gradient of $w$. In particular, $g_w = g_v$ a.e.\@ on $E$. Switching the role of $u$ and $v$, we obtain that $g_w = g_u$ a.e.\@ on $E$. Consequently, $g_u = g_w = g_v$ a.e.\@ on $E$.

If now $v\equiv c$ is a constant function, then $g_v = 0$, completing the proof.
\end{proof}
%
%
%  REFERENCES
%
%

\end{document}